%% file: KK_main_final_Ver2.tex
\documentclass{amsart}

\usepackage{lineno,hyperref,amsmath, amssymb,amsthm}

\theoremstyle{plain} 
\newtheorem{theorem}{\indent\bf Theorem}[section]
\newtheorem{lemma}[theorem]{\indent\bf Lemma}
\newtheorem{corollary}[theorem]{\indent\bf Corollary}
\newtheorem{proposition}[theorem]{\indent\bf Proposition}
\newtheorem{claim}[theorem]{\indent\bf Claim}
\theoremstyle{definition} 
\newtheorem{definition}[theorem]{\indent\bf Definition}
\newtheorem{remark}[theorem]{\indent\bf Remark}
\newtheorem{example}[theorem]{\indent\bf Example}

\DeclareMathOperator{\Div}{div}

\DeclareMathOperator{\oHess}{\overline{Hess}}

\DeclareMathOperator{\oRm}{\overline{Rm}}
\DeclareMathOperator{\inj}{inj}

\DeclareMathOperator{\Vol}{Vol}

\newcommand{\nab}{\nabla}
\newcommand{\Nab}{\overline{\nabla}}
\newcommand{\p}{\partial}

\newcommand{\Z}{\mathbb{Z}}

\newcommand{\R}{\mathbb{R}}
\newcommand{\C}{\mathbb{C}}

\newcommand{\nap}{\nabla^{\perp}}

\newcommand{\oR}{\overline{R}}

\newcommand{\ui}{\underline{i}}
\newcommand{\uj}{\underline{j}}
\newcommand{\uk}{\underline{k}}

\newcommand{\um}{\underline{m}}

\newcommand{\dmu}{d\mu}

\newcommand{\wnab}{\widetilde{\nab}}

\begin{document}
\pagestyle{plain}

\title[Stabilities and generalized Lagrangian mean curvature flow]{Hamiltonian stability for weighted measure and generalized Lagrangian mean curvature flow} 

\author{Toru Kajigaya and Keita Kunikawa}
\address{National Institute of Advanced Industrial Science and Technology (AIST), MathAM-OIL, 
Sendai 980-8577, Japan}
\email{kajigaya.tr@aist.go.jp}
\address{Advanced Institute for Materials Research, Tohoku University (AIMR), Sendai 980-8577, Japan
}
\email{keita.kunikawa.e2@tohoku.ac.jp}





%

\date{\today}
\maketitle

\begin{abstract}
In this paper, we generalize several results for the Hamiltonian stability and the mean curvature flow of Lagrangian submanifolds in a K\"ahler-Einstein manifold to more general K\"ahler manifolds including a Fano manifold equipped with a K\"ahler form $\omega\in 2\pi c_1(M)$ by using the methodology proposed by T. Behrndt \cite{Beh}.  Namely, we first consider a weighted measure on a Lagrangian submanifold $L$ in a K\"ahler manifold $M$ and investigate the variational problem of $L$ for the weighted volume functional. We call a stationary point of the weighted volume functional $f$-minimal, and define the notion of Hamiltonian $f$-stability as a local minimizer under Hamiltonian deformations. We show such examples naturally appear in a toric Fano manifold. Moreover, we consider the generalized Lagrangian mean curvature flow in a Fano manifold which is introduced by Behrndt and Smoczyk-Wang. We generalize the result of H. Li, and show that if the initial Lagrangian submanifold is a small Hamiltonian deformation of an $f$-minimal and Hamiltonian $f$-stable Lagrangian submanifold, then the generalized MCF converges exponentially fast to an $f$-minimal Lagrangian submanifold.
 \end{abstract}

\input{S1_intro}
\input{S2_hamstab}
\input{S3_example}

\input{S4_glmcf}

\input{S5_estimate}

\input{S67_converge}

\input{S_bib}
\end{document}

%% file: S1_intro.tex
\section{Introduction}
Let $(M, \omega)$ be a symplectic manifold of ${\rm dim}_{\R}M=2n$, where $\omega$ is the symplectic form.  An $n$-dimensional submanifold $L$ in $M$ is called {\it Lagrangian} if $\omega|_L=0$, and Lagrangian submanifolds are investigated by several geometric motivations. When $M$ admits a Riemannian metric,  variational problems for  Lagrangian submanifolds are of particular interests. For example, Harvey-Lawson introduced the notion of special Lagrangian submanifolds in Calabi-Yau manifolds \cite{HL}, which is a submanifold with volume-minimizing property with respect to the Ricci-flat K\"ahler metric. In the Calabi-Yau setting, the special condition of $L$ is actually equivalent to the {\it minimality} of $L$, namely, $L$ is a critical point of the volume functional, and the minimality is described by vanishing mean curvature. See \cite{Joyce} for further developments and some construction methods of special Lagrangian submanifold in a Calabi-Yau manifold. 

On the other hand, the situation is different when $M$ admits a K\"ahler structure with K\"ahler metric $g$ of positive Ricci curvature.   In fact, Lawson-Simons \cite{LS} proved that, in the complex projective space $\mathbb{C}P^n$ with the Fubini-Study metric, a stable submanifold for the volume functional is actually a complex submanifold, and hence, any minimal Lagrangian submanifold in $\mathbb{C}P^n$ is never stable in the usual sense because the Lagrangian condition implies the tangent space of $L$ does not contain any complex distribution. In such a remarkable situation, Y.-G. Oh studied a variational problem for Lagrangian submanifold in K\"ahler-Einstein manifold under restricted deformations \cite{Oh1, Oh2, Oh3}. Namely, he considered Hamiltonian deformations of $L$, which is a fundamental concept in symplectic geometry, and proved that there exist several examples of stable minimal Lagrangian submanifolds under Hamiltonian deformations. We call such a stable submanifold {\it Hamiltonian-stable Lagrangian submanifold}. See \cite{MO} and references therein for more examples of minimal and Hamiltonian-stable Lagrangian submanifold in a K\"ahler-Einstein manifold.

The {\it mean curvature flow} ({\it MCF} for short) is a fundamental concept to seek a minimal submanifold. It is known due to Smoczyk \cite{Smo3} that the MCF of Lagrangian submanifold $L$ in a K\"ahler-Einstein manifold preserves the Lagrangian condition, and hence, we call the flow \textit{Lagrangian mean curvature flow} (\textit{LMCF} for short).  Since MCF is a negative $L^2$ gradient flow of the volume functional of a submanifold, we naively expect that LMCF converges to a minimal and some kind of stable Lagrangian.  In fact, Thomas-Yau conjectured long-time existence and convergence of zero-Maslov LMCF in a Calabi-Yau manifold under some stability condition \cite{TY}. 
After that, Joyce gave further studies in Calabi-Yau setting and improved Thomas-Yau conjecture (see \cite{Joyce2}, \cite{Nev1} and references therein for the details). Although this conjecture is still widely open, some relevant results are known under restricted conditions. In \cite{Smo3} and \cite{SmW}, Smoczyk and Smoczyk-Wang proved long-time existence and convergence results of LMCF in a flat space under some convexity conditions. 


More generally, as pointed out by Smoczyk \cite{Smo2}, the LMCF in a K\"ahler-Einstein manifold generates a Hamiltonian deformation of $L$ whenever the mean curvature form of the initial Lagrangian $L$ is exact (or zero-Maslov). Thus, it is natural to ask whether the LMCF of exact Lagrangian exists for all time and converges to a Hamiltonian-stable minimal Lagrangian submanifold.
In this direction, there are still few convergence results of LMCF in general K\"ahler-Einstein manifolds (including Ricci-positive case).  For examples, Wang showed a convergence results for a graph \cite{Wan1}, and Li  showed some stability results of LMCF, that is, LMCF with exact mean curvature form converges to a minimal Lagrangian if the initial data is sufficiently close to a Hamiltonian-stable minimal Lagrangian \cite{Li}. 

So far, many studies of the Hamiltonian-stability of Lagrangian submanifolds and Lagrangian MCF are in the K\"ahler-Einstein setting and the volume is measured by the K\"ahler metric. However, the K\"ahler-Einstein condition is very restrictive despite of the facts that the concepts of Lagrangian submanifolds and Hamiltonian deformations are purely symplectic.  
Recently, Smoczyk-Wang \cite{SmW} introduced the notion of {\it Einstein connection}, and showed that some results  can be extended to  a Lagrangian submanifold $L$ in an almost K\"ahler manifold with Einstein connection when we consider {\it generalized mean curvature} instead of the mean curvature of $L$.  For example, any cotangent bundle $T^*N$ of a manifold $N$ admits a canonical Einstein connection, and Smoczyk-Tsui-Wang \cite{STW} proved a convergence result for generalized mean curvature flow. On the other hand, Lotay-Pacini \cite{LP} gave another general framework based on the results of Oh, in which they extended several results to {\it totally real} submanifolds in an almost Hermitian manifold, and  an interesting notion, the $J$-volume for totally real submanifold, was studied in \cite{LP}. Our results in the present paper intersect some of these results, however, our formulation is slightly different from theirs and more suitable for our purpose.

For a general symplectic manifold, it is not heuristic to measure the volume by the (almost) K\"ahler metric, whereas the (almost) K\"ahler metric is a {\it canonical} choice of metric and it is of course the best choice in a K\"ahler-Einstein manifold. In this direction, Behrndt introduced the notion of {\it almost-Einstein} K\"ahler manifold, that is, the Ricci form $\rho$ satisfies $\rho=C\omega+ndd^cf$ for some constant $C\in \mathbb{R}$ and function $f\in C^{\infty}(M)$ and studied  Lagrangian submanifolds in an {\it almost} Calabi-Yau manifold as an extension of the theory of Calabi-Yau manifold  \cite{Beh}. In his studies, he pointed out that to use a globally conformal K\"ahler metric or a weighted measure on $M$ is more suitable to investigate Lagrangian submanifolds in an almost-Einstein K\"ahler manifold.  Furthermore, the first author showed that it is natural to consider the weighted metric or measure in order to describe some geometric properties of Lagrangian submanifold in a K\"ahler quotient space (see \cite{Kaji}).

Notice that the almost-Einstein manifold is still restrictive and is actually included in the more general notion due to Smoczyk-Wang (see \cite{SmW}). However, the K\"ahler condition has an advantage in many situations, and almost-Einstein K\"ahler manifolds contain some important classes of K\"ahler manifolds, namely, an almost Calabi-Yau manifold and a Fano manifold equipped with a K\"ahler form $\omega\in 2\pi c_1(M)$, where $c_1(M)$ is the first Chern class.  Moreover, a symplectic manifold $(M, \omega)$ with compatible almost complex structure $J$ satisfying $\omega\in \lambda c_1(M)$ for some positive constant $\lambda$ is called {\it monotone} and specifically studied in symplectic geometry (See \cite{Oh4}).
Besides the work of Behrndt, it is still unclear how the successful results in K\"ahler-Einstein manifolds is extended to Fano manifolds or more general class of manifolds by using Behrndt's methodology.  One of purpose in the present paper is to  confirm this question and exhibit the extended results.


 We briefly summarize our main results in the present paper.  Let $(M, \omega, J)$ be a K\"ahler manifold and $L$  a Lagrangian submanifold in $M$. We denote the compatible K\"ahler metric by $g$. For a function $f\in C^{\infty}(M)$,  we consider a weighted measure $dv_f=e^{nf}dv_L$ on $L$, where $dv_L$ is the volume form w.r.t. $g$, and investigate the variational problem of $L$ for the weighted volume functional ${\rm Vol}_{g_f}(L):=\int_Ldv_f$. We call  a Lagrangian submanifold  {\it $f$-minimal} if it is a critical point of ${\rm Vol}_{g_f}$ under any infinitesimal deformations. By the first variational formula, $L$ is $f$-minimal if and only if $H_f:=e^{-2f}(H-n\Nab f^{\perp})=0$, where $H_f$ is the mean curvature vector of $L$ w.r.t. $g_f$. Instead of using $H_f$, we define a modified vector
 \begin{align*}
 K:=H-n\Nab f^{\perp}.
 \end{align*}
Following Behrndt \cite{Beh} and Smoczyk-Wang \cite{SmW}, we call $K$ {\it generalized mean curvature vector} of $L$ for the weight function $f$. By definition, the $f$-minimality is equivalent to $K=0$.  Moreover, an $f$-minimal Lagrangian submanifold is called {\it Hamiltonian $f$-stable} if the second variation of ${\rm Vol}_{g_f}$ is non-negative under any Hamiltonian deformations.  Obviously, these notions coincides with the usual ones considered in \cite{Oh1, Oh2} when $f=0$.


If the K\"ahler form $\omega$ of $M$ satisfies $\rho=C\omega+ndd^cf_{\omega}$ for some function $f_{\omega}\in C^{\infty}(M)$,  the potential function $f_\omega$ is a reasonable choice of the weight function. In fact, some $f_\omega$-minimal Lagrangian submanifolds naturally appears as a generalization of minimal Lagrangian submanifolds in a K\"ahler-Einstein manifold. Moreover, we show the generalized mean curvature form $\alpha_{K}:=(i_{K}\omega)|_{L}$ of arbitrary Lagrangian submanifold $L$ is a closed form and the de-Rham cohomology class $[\alpha_{K}]$ is a Hamiltonian invariant of $L$, namely, $[\alpha_{K}]$ is regarded as a generalization of the classical Maslov class (See Section 2.4). In particular, it turns out that any $f_\omega$-minimal Lagrangian submanifold in $M$ is monotone in the sense of \cite{Oh4}.  Furthermore, we give a simple criterion for the Hamiltonian $f_\omega$-stability of $f_\omega$-minimal Lagrangian submanifold as a generalization of the result of Chen-Leung-Nagano \cite{CLN} and Oh \cite{Oh1}, that is,  the Hamiltonian $f_\omega$-stability is equivalent to 
\begin{align*}
\lambda_1(\Delta_{f_\omega})\geq C, 
\end{align*}
where $\lambda_1$ is the first eigenvalue of the weighted Laplace operator $\Delta_{f_\omega}$ acting on $C^{\infty}(L)$ (Theorem \ref{hsta}). 

For example, we consider a torus orbit in a weighted projective space $\mathbb{C}P^n_{\bf a}$ for positive integers ${\bf a}:=(1, a_2,\ldots, a_{n+1})$. Since $\mathbb{C}P^n_{\bf a}$ is obtained by a weighted $S^1$-Hamiltonian action on  $\mathbb{C}^{n+1}$, the standard K\"ahler reduction yields a canonical K\"ahler structure $(\omega_{\bf a}, J_{\bf a})$ on $\mathbb{C}P^n_{\bf a}$, and $\omega_{\bf a}$ defines a canonical potential function $f_{\bf a}$ (See Section 3). We prove that the reduced manifold $T^{n+1}/S^1$ of the Clifford torus $T^{n+1}$ in $\mathbb{C}^{n+1}$ by the weighted $S^1$-action is an $f_{\bf a}$-minimal and Hamiltonian $f_{\bf a}$-stable Lagrangian submanifold in $\mathbb{C}P^n_{\bf a}$ (Theorem \ref{clifford}). This generalizes the theorem by Oh \cite{Oh1} in which he proved the statement for the Clifford torus in the complex projective space $\mathbb{C}P^n$.

In the latter half of this paper, we consider \textit{generalized Lagrangian mean curvature flow} ({\textit{GLMCF} for short) in a compact K\"ahler manifold with $\rho=C\omega+ndd^cf_\omega$, namely, a family of immersions $F: L\times [0, T) \rightarrow M$ satisfying   
\begin{align*}
\frac{\p F}{\p t}(x, t)=K(x, t), \ \ F(x, 0)=\phi(x),  
\end{align*}
where $\phi: L\rightarrow M$ is a Lagrangian immersion. This flow is proposed by Behrndt \cite{Beh} and he showed that the flow preserves Lagrangian condition (see also \cite{SmW}). Moreover if the initial Lagrangian is exact, i.e., $[\alpha_{K_0}]=0$, then we prove that GLMCF preserves exactness of $\alpha_{K_t}$. In particular, GLMCF generates a Hamiltonian deformation of the exact initial data. This is an extension of the results by Smoczyk in K\"ahler-Einstein manifolds. These facts lead us to an expectation similar to the K\"ahler-Einstein case, that is, the GLMCF may converge to an $f_\omega$-minimal and Hamiltonian $f_\omega$-stable Lagrangian.  

In the present paper, we confirm this question under some initial conditions. More precisely, if an initial Lagrangian $L$ is exact, $\lambda_1(\Delta_{f_\omega})>C$ and has sufficiently small weighted $L^2$ norm of $K$, then the GLMCF with initial data $L$ exists for all time and converges exponentially fast to an $f_\omega$-minimal Lagrangian with $\lambda_1(\Delta_{f_\omega})>C$ (Theorem \ref{cvthm}). Examples satisfying the initial conditions of this theorem are given by a small Hamiltonian deformation of some torus orbits constructed above. 

On the other hand, there exists an $f_\omega$-minimal Lagrangian with $\lambda_1(\Delta_{f_\omega})=C$ and this is not the case of Theorem \ref{cvthm}. In order to deal with this case, we consider small Hamiltonian deformation of $f_\omega$-minimal Lagrangian with $\lambda_1(\Delta_{f_\omega})=C$ as an initial data of the flow. Then we prove that the GLMCF exists for all time and converges exponentially fast to an $f_\omega$-minimal Lagrangian (Theorem \ref{main2}).  The main techniques we used in our proof are generalization of the ones developed by Li \cite{Li} in K\"ahler-Einstein manifolds.

The paper is organized as follows. In Section 2, we define notions of $f$-minimality and Hamiltonian $f$-stability for Lagrangian submanifold $L$ in a K\"ahler manifold $M$ with weighted measure, and derive the first and second variational formula of the weighted volume functional under Hamiltonian deformations of $L$. In Section 3, we give examples of $f$-minimal and Hamiltonian $f$-stable Lagrangian torus orbits in the weighted projective spaces. In Section 4, we introduce the generalized Lagrangian mean curvature flow and show some fundamental properties of GLMCF.  In Section 5, we give several estimates for GLMCF. Finally, in Section 6 and 7, we prove the convergence results of GLMCF.

\subsection*{Acknowledgements} 
The authors would like to thank Shouhei Honda and Hikaru Yamamoto for helpful comments. 
A part of this work was done while the first author was staying at University of T\"ubingen by the JSPS Program for Advancing Strategic International Networks to Accelerate the Circulation of Talented Researchers, Mathematical Science of Symmetry, Topology and Moduli, Evolution of International Research Network based on OCAMI. He is grateful for the hospitality of the University. 
The second author is supported by the Grant-in-Aid for JSPS Fellows (16J01498). During the preparation of this paper he has stayed in Max Planck Institute for Mathematics in the Sciences, Leipzig. He is grateful to J\"urgen Jost for the hospitality. 

%% file: S2_hamstab.tex
\section{Hamiltonian stabilities for Lagrangian submanifolds}
In this section, we introduce notions of $f$-minimality and Hamiltonian $f$-stability of Lagrangian submanifold in a K\"ahler manifold with weighted measure.

\subsection{Hamiltonian minimality and stability for weighted measure}
Let $(M, \omega, J)$ be a complex $n$-dimensional (almost) K\"ahler manifold, where $\omega$ is the symplectic form and $J$ is the (almost) complex structure. We define the compatible  metric $g$ by $\omega(\cdot, \cdot):=g(J\cdot,\cdot)$. Consider a Lagrangian immersion $\phi:L\rightarrow M$ of a manifold $L$, that is, $\phi^*\omega=0$ and $\dim_{\R}L=n$. Then, we have the following isomorphisms which we shall use throughout this article:
\begin{align}\label{isomorphism}
T_pL\ \tilde{\rightarrow}\ T_p^{\perp}L\ \tilde{\rightarrow}\ T_p^*L, \quad V\mapsto JV\mapsto \alpha_V:=\phi^*(i_V\omega)
\end{align} 
for any $p\in L$, where $T_p^{\perp}L$ is the normal space of $T_pL$ in $T_pM$ w.r.t. $g$ and $i$ denotes the inner product. For a smooth function $f\in C^{\infty}(M)$, we consider a globally conformal K\"ahler metric $g_f:=e^{2f}g$ of $g$. 
The induced metric $\phi^*g$ on $L$ is denoted by the same symbol $g$.  Also, the pull-back $\phi^*f$ is denoted by $f\in C^\infty (L)$ if there is no confusion from contexts. Denote the volume form on $L$ w.r.t. $g$ and $g_f$ by $d\mu$ and $d\mu_f$, respectively. Note that we have $d\mu_f=e^{nf}d\mu$. In the following,  we  consider a weighted metric measure space $(L, g, d\mu_f=e^{nf}d\mu)$.

 We shall consider a variational problem on a weighted metric measure space $(L, g, d\mu_f)$ w.r.t. a weighted volume functional 
 \begin{align*}
 \Vol_f(\phi):=\int_Ld\mu_f=\int_Le^{nf}d\mu.
 \end{align*}
  It is easy to see that the first variational formula is given by
 \begin{align}\label{firstv}
&\frac{d}{ds}\Big{|}_{s=0}\Vol_f(\phi_s)=-\int_{L}g(K, V)d\mu_f,\quad {\rm where}\quad K:=H-n(\Nab f)^\perp
\end{align}
for any infinitesimal deformation $\phi_s$ of $\phi=\phi_0$ and $H$ is the mean curvature vector of $\phi$ w.r.t. $g$. Note that the mean curvature vector of $\phi$ w.r.t. $g_f$ is given by $H_f=e^{-2f}\{H-n(\Nab f)^{\perp}\}$. In particular, $\phi$ is a critical point of the weighted volume functional if and only if $K=0$. 

Moreover, we set
\begin{align*}
\alpha_{K}:&=\phi^*(i_{K}\omega)=\alpha_H-n\phi^*d^cf.
\end{align*}
According to \cite{Beh} and \cite{SmW}, we call $K$ and $\alpha_{K}$ \textit{generalized mean curvature vector} and \textit{generalized mean curvature form}, respectively.

We shall consider a variational problem for a restricted deformation.  
A smooth deformation $\phi_s:(-\epsilon, \epsilon)\times L\rightarrow M$ of a Lagrangian immersion $\phi_0=\phi:L\rightarrow (M, \omega)$ into a symplectic manifold is called \textit{Hamiltonian deformation} if the variational vector field $V_s:=(d/ds)\phi_s$ is a Hamiltonian vector field along $\phi_s$ for $s\in (-\epsilon, \epsilon)$, i.e., there exists $u_s\in C^\infty(L)$ so that $\alpha_{V_s}=du_s$ for each $s$.  By Cartan's formula, it is easy to see that the Hamiltonian deformation preserves the Lagrangian property.

The following definitions are extensions of the notions in \cite{Oh1} and \cite{Oh2}. 
\begin{definition}
Let $\phi:L\rightarrow M$ be a Lagrangian immersion into an (almost) K\"ahler manifold. We consider a weighted volume functional ${\rm Vol}_{g_f}$ for $f\in C^\infty(M)$. 
\begin{enumerate}
\item $\phi$ is called \textit{$f$-minimal} if $\phi$ is a critical point of ${\rm Vol}_{g_f}$ under any infinitesimal deformation of $\phi$, i.e., $K=H_f\equiv 0$, or equivalently, $\alpha_{K}=0$.

\item $\phi$ is called \textit{Hamiltonian $f$-minimal} if $\phi$ is a critical point of ${\rm Vol}_{g_f}$ under any Hamiltonian deformation of $\phi$. 

\item A Hamiltonian $f$-minimal $\phi$ is called \textit{Hamiltonian $f$-stable} if the second variation of ${\rm Vol}_{g_f}$ is nonnegative for any Hamiltonian deformation of $\phi$. 
\end{enumerate}
\end{definition}




We shall derive the Euler-Lagrange equation for the Hamiltonian minimal Lagrangian submanifolds. The following notations are more suitable for the Bakry-\'Emery calculation. Namely, we define a \textit{weighted co-differential} $\delta_f\alpha$ acting on $\Omega^1(L)$ by 
\begin{align*}
\delta_f\alpha:=-e^{-nf}\Div(e^{nf}\alpha^\sharp)=\delta\alpha-ng(df, \alpha), \ \ \alpha\in\Omega^1(L), 
\end{align*}
where $\alpha^\sharp\in\Gamma(TL)$ is the counterpart of $\alpha \in \Omega^1(L)$ by the canonical isomorphism via the metric $g$ on $L$. By the divergence theorem, $\delta_f$ is in fact the co-differential of $d$ w.r.t. $d\mu_f$: 
\begin{align*}
\int_L(\delta_f\alpha)ud\mu_f=\int_Lg(\alpha, du)d\mu_f, \ \ u\in C^\infty_c(L).
\end{align*}
For a smooth function $u\in C^\infty(L)$, we define a weighted Laplacian $\Delta_f$ acting on $C^\infty(L)$ by 
\begin{align*}
\Delta_fu:=-\delta_fdu=e^{-nf}\Div(e^{nf}\nab u)=\Delta u+ng(\nab f, \nab u), \ \ \ \  f\in C^\infty(L),  
\end{align*}
where $\Delta=\Div \nab=-\delta d$ is the Laplace-Beltrami operator, i.e., negative Hodge-de Rham Laplacian acting on $C^\infty(L)$ and $n=\dim_{\R}L$. By the divergence theorem, we have the following formula:
\begin{align}
\int_L(\Delta_f u_1)u_2 d\mu_f=-\int_Lg(du_1, du_2)d\mu_f=\int_L u_1(\Delta_f u_2) d\mu_f, \label{stokes}
\end{align}
for any compactly supported functions $u_1, u_2\in C^\infty_c(L)$.

\begin{proposition}\label{EL}
A Lagrangian immersion $\phi: L\rightarrow M$ into an almost K\"ahler manifold is Hamiltonian $f$-minimal if and only if 
\begin{align*}
\delta_f\alpha_K=0.
\end{align*}
\end{proposition}
\begin{proof}
Suppose $V:=(d/ds)|_{s=0}\phi_s$ is a normal vector field. Since $V$ is a Hamiltonian vector field along $L$, there exists a function $u\in C^\infty(L)$ so that $\alpha_V=\phi^*(i_X\omega)=du$. Then the general first variation formula implies \begin{align*}
\frac{d}{ds}\Big|_{s=0}\Vol_f(\phi_t)&=-\int_Lg_f(H_f, V)d\mu_f=-\int_Lg(K, V)d\mu_f\\
&=-\int_Lg(\alpha_K, du)d\mu_f=-\int_L(\delta_f\alpha_K)ud\mu_f. 
\end{align*}
Since $u\in C^\infty(L)$ is arbitrary, the notion of Hamiltonian minimality is equivalent to $\delta_f\alpha_K=0$. 
\end{proof}

A Hamiltonian minimality for Lagrangian submanifolds in almost Calabi-Yau manifolds has already been introduced by Yamamoto in \cite{Yam}. In the almost Calabi-Yau setting, our definition is equivalent to his definition under a suitable choice of  g.c.K. metric (see below). 

The following proposition shows that we obtain a Hamiltonian $f$-minimal Lagrangian submanifold as a compact orbit of a Lie group action:

\begin{proposition}\label{homoge}
Let $(M,\omega, J)$ be a K\"ahler manifold and $G$ a connected compact subgroup of $\textup{Aut}(M, \omega, J)^0$ an identity component of the automorphism group of $(M,\omega, J)$. Suppose the $G$-action admits a Lagrangian orbit.  Then, any compact Lagrangian orbit is Hamiltonian $f$-minimal for any $G$-invariant function $f\in C^\infty(M)$.
\end{proposition}

\begin{proof}
Since $f$ is $G$-invariant, so is $g_{f}=e^{2f}g$. Thus the generalized mean curvature form $\alpha_K$ of the Lagrangian orbit $\mathcal{O}=G\cdot p$ is an $G$-invariant $1$-form, and hence, $\delta_{f}\alpha_K$ is an $G$-invariant function. Since $G$ acts on $\mathcal{O}$ transitively, the function is constant on $L$. Moreover, by the compactness of $L$, the divergence theorem implies 
\begin{align*}
\int_{\mathcal{O}}\delta_{f}\alpha_Kd\mu_f=0, 
\end{align*}
and hence, $\delta_{f}\alpha_K=0$. This proves (1). 
\end{proof}

\subsection{Second variation formula}\label{2ndv}
In the following, we always assume $(M,\omega, J)$ is a K\"ahler manifold.  First, we derive a general second variational formula for a Lagrangian immersion.

 Let $\Nab$ and $\nab$ be Levi-Civita connections on $TM$ and $TL$ w.r.t. $g$, respectively. We adopt the following definition for the curvature tensor on $TM$ by
\begin{align*}
\oR(X, Y)Z:=\Nab_X\Nab_YZ-\Nab_Y\Nab_XZ-\Nab_{[X, Y]}Z, 
\end{align*}
for $X, Y, Z\in\Gamma(TM)$. We also define a $(0, 4)$ curvature tensor by 
\begin{align*}
\oRm(Z, W, X, Y):=g(Z, R(X, Y)W). 
\end{align*}
Other curvature tensors will be defined likewise. We denote the Ricci form of $M$ by $\rho$. 

For a Lagrangian immersion $\phi: L\rightarrow M$, the second fundamental form is denoted by $B$, and the mean curvature vector is given by $H:={\rm tr}B$. 

\begin{proposition}\label{2ndvf}
Let $\phi:L\rightarrow M$ be a Lagrangian immersion into a K\"ahler manifold.  Then, for any normal deformation $\{\phi_s\}$ with $V=(d/ds)|_{s=0}\phi_s\in \Gamma(T^{\perp} L)$, we have 
\begin{align*}
\frac{d^2}{ds^2}\Big|_{s=0}&\Vol_f(\phi_s)=\int_L\{|\delta_f\alpha_V|^2-\rho_f(V,JV)-g(B(JV, JV), K)-g(\Nab_VV,K) +g(V,K)^2\}d\mu_f, 
\end{align*}
where ${\rho}_f:=\rho-ndd^cf$. 
\end{proposition}

A direct consequence is the following:
\begin{corollary}[cf. \cite{Le}]\label{le}
Suppose furthermore $\rho_f\leq 0$. Then, any $f$-minimal Lagrangian submanifold in $M$ is $f$-stable.
\end{corollary}

Note that this corollary is in fact a special case of the result by L$\hat{\rm e}$ \cite{Le}.

In the following, we give a proof of Proposition \ref{2ndvf}.

Let $\phi:L\rightarrow M$ be a Lagrangian immersion of an $n$-dimensional manifold $L$. Consider a smooth deformation $\phi_s:L\times (-\eta, \eta) \rightarrow M$ of the Lagrangian immersion $\phi=\phi_0$. We regard $L$ as a Riemannian manifold endowed with the time dependent induced metric $g_s=\phi^*_sg$. Let $\{\p_i\}_{i=1}^n$ be a normal coordinate frame in $TL$ around a point $p\in L$ w.r.t. $g_{0}$, that is, 
\begin{align*}
(g_{0})_{ij}(p)=g_{0}(\p_i, \p_j)(p)=\delta_{ij}, \ \ \nab_{\p_i}\p_j(p)=0.   
\end{align*}
For simplicity, $(g_s)_{ij}$ may be just written as $g_{ij}$. 
Using this normal coordinate frame, we obtain a time dependent frame along $L_s=\phi_s(L)$ by
\begin{align*}
e_i:=d\phi_s(\p_i), \ \ \ 1\leq i\leq n. 
\end{align*}
Note that this time dependent frame $\{e_i\}_{i=1}^n$ is not orthonormal except at a point $x=\phi_0(p)\in L_{0}$. Since $\{e_i\}_{i=1}^n$ is a coordinate frame, they commute with the variation vector field $d\phi_s((d/ds))=V_s$, that is, 
\begin{align*}
[V, e_i]=0, \ \ 1\leq i\leq n. 
\end{align*} 
We take a frame in $T^\perp L_s$ by $\{Je_i\}_{i=1}^n$, then $\{e_1, \cdots e_n, Je_1, \cdots, Je_n\}$ is a frame in $TM$ along $L_s$ which is orthonormal only at a point $x=\phi_0(p)\in L_0$. 
Now, we are ready to compute the second variation formula. First, we just take a time derivative of the first variation formula: 
\begin{align}
\frac{d^2}{ds^2}\Vol_f(\phi_s)&=-\frac{d}{ds}\int_{L_s}g(K, V)d\mu_f=\int_{L_s}-\frac{d}{ds}g(K, V)+g(K, V)^2d\mu_f.  \label{var1}
\end{align}
We  shall compute $-(d/ds)g(K, V)$. In the frame $\{e_i, Je_i\}$, we set 
\begin{align*}
h^{k}_{ij}&:=g(JB(e_i, e_j), e_k).
\end{align*}
It is well-known that $h^{k}_{ij}$ is symmetric for all components by the integrability of $J$. Moreover, we set
\begin{align*}
H^k&:=\alpha_H(e_k)=g(JH, e_k)=g^{ij}h^k_{ij},\quad K^k:=\alpha_K(e_k),\ {\rm and}\ V^k:=\alpha_V(e_k),
\end{align*}
where $g^{ij}$ denotes the inverse matrix of $g_{ij}$ and we use the summation convention for repeated indices. Note that, for instance, we have $JH=g^{kj}H^ke_j$ in this notation.  Then the first term of the integrand in \eqref{var1} can be written as 
\begin{align}
\frac{d}{ds}g(K, V)&=\frac{d}{ds}g(JK, JV)=\frac{d}{ds}\Big(g^{ij}K^iV^j\Big)\nonumber\\
&={\Big(\frac{d}{ds}g^{ij}\Big)K^iV^j}+{\Big(\frac{d}{ds}K^i\Big)V^i}+{K^i\Big(\frac{d}{ds}V^i\Big)} \label{var2}
\end{align}
at a point $x\in L_0$. 
\begin{lemma}[Lemma 3.2 in \cite{Oh2}]
For a normal deformation $\phi_s$, we have
\begin{align*}
\frac{d}{ds}g_{ij}&=-2V^kh^k_{ij}\quad {\rm and}\quad \frac{d}{ds}g^{ij}=2g^{ip}(V^ah_{pq}^a)g^{qj}.
\end{align*}
\end{lemma}
By using this lemma,  the first term in \eqref{var2} is computed by
\begin{align}
\Big(\frac{d}{ds}g^{ij}\Big)K^iV^j=2g(B(JV, JV), K). \label{first}
\end{align}
Moreover, the third term becomes
\begin{align}
{K^i\Big(\frac{d}{ds}V^i\Big)}&=K^i\Nab_Vg(JV, e_i)=K^ig(\Nab_VJV, e_i)+K^ig(JV, \Nab_Ve_i)\label{third}\\
&=g(\Nab_VV, K)-g(V, B(JK,JV))\nonumber\\
&=g(\Nab_VV, K)-g(B(JV,JV), K)\nonumber\nonumber
\end{align}
since $\Nab_Ve_i=\Nab_{e_i}V$ and the symmetry of $h_{ij}^k$ .  For the second term, we use the following lemma:

\begin{lemma} \label{Ldsa}
 Let $M$ be a K\"aler manifold, $\phi=\phi_0:L\rightarrow M$ a Lagrangian immersion and $L_s=\phi_s(L)$  a smooth deformation with $V=(d/ds)\phi_s$. Then the generalized mean curvature form $\alpha_{K_s}$ along $L_s$ satisfies 
\begin{align}
\frac{d}{ds}K^i&=\nab_i(\nab_jV^j+nV^j\nab_jf)+V^j(\oR_f)_{ji},\ {\rm or\ equivalently}  \label{dsaKF}\\
\frac{d}{ds}\alpha_K&=-d\delta_f\alpha_V+\phi^*_s(i_V{\rho}_f), \label{dsaK}
\end{align}
where $(\oR_f)_{ji}:=-{\rho}_f(Je_j, e_i)$ and $(d/ds)\alpha_K$ is a $1$-form defined by 
\begin{align*}
\Big(\frac{d}{ds}\alpha_K\Big)(X):=\frac{d}{ds}(\alpha_K(X)), \ \ X\in\Gamma(TL). 
\end{align*}
\end{lemma}

\begin{proof}
By a similar computation to mean curvature flow (see \cite{SuY} and \cite{Wan}), we have 
\begin{align}
\frac{d}{ds}h_{ij}^k&=\nabla_{i}\nabla_{j}V^k-V^mh^m_{jl}h^k_{li}-V^mh^m_{kl}h^l_{ij}+V^m\oR_{\underline{m}i\underline{k}j}, \label{dth}
\end{align}
where $\nab_i\nab_jV^k=g(\nab_{e_i}\nab_{e_j}(JV), e_k)$, and $\oR_{\um i\uk j}=\oRm(Je_m, e_i, Je_k, e_j)$.  By the symmetry of $h^k_{ij}$, taking a trace with $g^{jk}$ in the above equation gives 
\begin{align}
\frac{d}{ds}H^i&=\frac{d}{ds}(g^{jk}h^k_{ij})\label{dsH}\\
&=\Big(\frac{d}{ds}g^{jk}\Big)h^k_{ij}+g^{jk}\Big(\frac{d}{ds}h^k_{ij}\Big) \nonumber\\
&=2V^mh^m_{jk}h^k_{ij}+\nab_i\nab_jV^j-V^mh^m_{jl}h^j_{li}-V^mh^m_{jl}h^l_{ij}+V^m\oR_{\underline{m}i\underline{k}k}\nonumber\\
&=\nab_i\nab_jV^j+V^m\oR_{mi}.\nonumber
\end{align}
In the forth equality, we have used $\oR_{\underline{m}i\underline{k}k}=\oR_{mi}$ which holds on a Lagrangian submanifold in a K\"ahler manifold. On the other hand, we have  
\begin{align}
-n\frac{d}{ds}\nap_if&=n\frac{d}{ds}g(\Nab f, Je_i)\label{dsf}\\
&=ng(\Nab_V\Nab f, Je_i)+ng(\Nab f, \Nab_V(Je_i))\nonumber\\
&=n\oHess_f(V, Je_i)+ng(\Nab f, \Nab_{e_i}(JV)) \nonumber\\
&=n\oHess_f(V, Je_i)+ne_i\{g(\nab f, JV)\}-n\oHess_f(JV, e_i)\nonumber\\
&=-ndd^cf(V,e_i)+ne_i\{g(\nab f, JV)\}.\nonumber
\end{align}
Thus, by \eqref{dsH} and  \eqref{dsf} we obtain 
\begin{align}
\frac{d}{ds}K^i&=\nab_i(\nab_jV^j+nV^j\nab_jf)+V^j({\oR}_f)_{ji} 
\end{align}
which proves  \eqref{dsaKF}. 
\end{proof}

By \eqref{dsaKF}, we see
\begin{align}
\Big{(}\frac{d}{ds}K^i\Big{)}V^i=-g(d\delta_f\alpha_V, \alpha_V)+{\rho}_f(V,JV). \label{second}
\end{align}

Substituting \eqref{first}, \eqref{third} and \eqref{second} to \eqref{var1}, we prove Proposition \ref{2ndvf}.

If we restrict our attention to Hamiltonian deformations, we obtain the following formula which generalizes the result in \cite{Oh2}:

\begin{corollary}\label{hmin}
Let $(M, \omega, J)$ be a K\"ahler manifold.  Suppose $\phi$ is Hamiltonian $f$-minimal and $\phi_s$ is a Hamiltonian deformation of $\phi$ so that $\alpha_{V}=du$ (i.e., $JV=\nabla u$) for some $u\in C^{\infty}(L)$.  Then, we have
 \begin{align*}
\frac{d^2}{ds^2}\Big|_{s=0}&\Vol_f(\phi_s)=\int_L\{|\Delta_fu|^2-\rho_f(\nabla u,J\nabla u)-2g(B(\nabla u, \nabla u), K)+g(J\nabla u,K)^2\}d\mu_f. 
\end{align*}
\end{corollary}
\begin{proof}
Since $V^i_s=\nabla_iu_s$, setting $w:=(du_s/ds)|_{s=0}$, the integrant of the third term in \eqref{var2} is computed by
\begin{align*}
\int_L K^i\Big(\frac{d}{ds}V^i\Big)d{\mu}_f=\int_L K^i\Big(\nabla_i\frac{du_s}{ds}\Big)d{\mu}_f=\int_L g(\alpha_K, dw)d\mu_f=\int_L \delta_f\alpha_K wd\mu_f=0.
\end{align*}
Here, we used $\delta_f\alpha_K=0$ by the assumption and Proposition \ref{EL}. Equivalently, we obtain
\begin{align*}
\int_Lg(\Nab_VV, K)d\mu_f=\int_Lg(B(JV,JV), K)d\mu_f
\end{align*} 
by \eqref{third}. Combining this with Proposition \ref{2ndvf},  we prove the formula.
\end{proof}

\subsection{Weight function}
Now, we suppose furthermore the Ricci form of $(M, \omega, J)$ satisfies 
\begin{align}\label{almE}
\rho=C\omega+ndd^cf,\quad {\rm or\ equivalently}\quad \rho_f=C\omega
\end{align}
for some constant $C\in \R$ and $f\in C^\infty(M)$, where $dd^c=2\sqrt{-1}\p\bar\p$ and $n=\dim_{\C}M$. We call the function $f$ a {\it weight function} for $\omega$.

 \begin{remark}\label{BE}
The condition \eqref{almE} has been considered in \cite{Beh} and \cite{SmW}, respectively. A K\"ahler manifold satisfying \eqref{almE} is called {\it almost Einstein} in \cite{Beh}, and a metric and complex connection (with torsion) on an almost K\"ahler manifold so called the {\it Einstein connection} is introduced as a generalization of \eqref{almE} in \cite{SmW}. If $C>0$, then \eqref{almE} implies $\omega\in C'c_1(M)$ for some positive $C'$ since the Ricci form $\rho$ of a K\"ahler manifold represents $2\pi c_1(M)$, where $c_1(M)$ is the first Chern class of $(M,J)$. Such a symplectic manifold is  called {\it monotone} in symplectic geometry (cf. \cite{AM}). 

 On the other hand, in the Bakry-\'Emery theory, the {\it Bakry-\'Emery Ricci tensor} on a weighted metric measure space $(M,g, e^{2nf}d\mu)$ is defined by 
 \begin{align}\label{BEric}
\overline{\rm Ric}_f:=\overline{\rm Ric}-2n\overline{\rm Hess}_f
 \end{align} 
 for some $f\in C^{\infty}(M)$ .  We remark that $\rho_f:=\rho-ndd^cf$ does not coincide with $\overline{\rm Ric}_f(J,\cdot, \cdot)$ in general. However, if $M$ satisfies an additional assumption, then we have $\rho_f=\overline{\rm Ric}_f(J\cdot, \cdot)$. See  below.
 \end{remark}

A typical example  satisfying \eqref{almE} is given by Fano manifold, i.e., a  compact complex manifold $(M,J)$ of positive first Chern class $c_1(M)$.  In fact, we can take a K\"ahler form so that  $\omega'\in 2\pi c_1(M)$, then there exists a real function $f$ satisfying $\rho=\omega'+ndd^cf$ since $\rho$ represents $2\pi c_1(M)$. By rescaling $\omega'$ if necessary, we obtain the desired $\omega$ for some positive $C$. Note that, for fixed $J$ and the K\"ahler class $[\rho]=2\pi c_1(M)$, the compatible symplectic forms in $[\rho]$ consist of an open and convex set (cf. \cite{Abreu}). 

An interesting case is when $(M, \omega, J)$ is the {\it K\"ahler-Ricci soliton} ({\it KRS} for short), i.e., there exists a non-trivial holomorphic vector field $X$ so that 
\begin{align}\label{KRS}
\rho=C\omega+\mathcal{L}_X\omega.
\end{align}
Note that a compact KRS exits only when $M$ is a Fano manifold. Moreover, if $M$ admits a KRS for non-trivial $X$, then there are no K\"ahler-Einstein metric on $M$ (See \cite{TiZhu}). The weight function on the KRS is given as follows (See \cite{BG2} for details): Since $X$ is holomorphic and $M$ is simply-connected,  there exists a real valued function $f\in C^{\infty}(M)$ so that  $X=\Nab f-\sqrt{-1}J\Nab f$, and we easily see $\mathcal{L}_X\omega=ndd^cf$, i.e., $f$ is a weight function. In this case, we have $\rho_f=\overline{\rm Ric}_f(J\cdot, \cdot)$. In fact,  since $X$ is holomorphic, we have $\mathcal{L}_{\Nab f}J=0$, and hence, $\overline{\rm Hess}_f(V,W)=\overline{\rm Hess}_f(JV,JW)$ for any real vector field $V,W\in\Gamma(TM)$. Therefore, we see
\begin{align}\label{ddc}
dd^cf(V, W)=-\overline{\rm Hess}_f(V,JW)+\overline{\rm Hess}_f(JV,W)=2\overline{\rm Hess}_f(JV,W).
\end{align}
Thus, $\rho_f=\overline{\rm Ric}_f(J\cdot, \cdot)$ when $M$ is a K\"ahler-Ricci soliton. 

\begin{example}\label{toric}
A complex $n$-dimensional compact K\"ahler manifold $(M,\omega, J)$ is called {\it toric} if a real $n$-dimensional torus $T^n$ acts on $M$ holomorphically and in a Hamiltonian way with moment map $\mu: M\rightarrow \mathbb{R}^n$. The well-known Delzant construction provides a canonical way to construct a toric K\"ahler manifold (see \cite{Abreu0} and references therein).  When $(M,\omega, J)$ is a toric Fano manifold, by the result of Wang-Zhu \cite{WZ},  there exists a unique  K\"ahler-Ricci soliton up to holomorphic automorphism on $M$, and the soliton is K\"ahler-Einstein if and only if the Futaki invariant vanishes.  In particular, we have several examples of toric Fano manifolds which do not admit any K\"ahler-Einstein metric (see \cite{WZ}).

When $(M,\omega, J)$ is a compact toric K\"ahler manifold,  any regular $T^n$-orbit is Lagrangian and just a level set $\mu^{-1}(c)$ for some $c\in \mathbb{R}^n$. Take any $T^n$-invariant function $f$. For example, we can take $f$ as the weight function for $\omega$ when $M$ is Fano.  Then, we see $T^n\subset {\rm Isom}(M, e^{2f}g)$ since $T^n\subset {\rm Aut}(M, \omega, J, g)$, and hence, there exists an orbit $\mathcal{O}$ with the maximum volume w.r.t. the metric $g_f=e^{2f}g$ among the regular $T^n$-orbits, and $\mathcal{O}$ must be minimal w.r.t. $g_f$, i.e., $f$-minimal (see Proposition 1 and Corollary 3 in \cite{Pac}). Because any regular $T^n$-orbit is Lagrangian, $\mathcal{O}$ is actually a  $f$-minimal Lagrangian submanifold in $M$. Note that any regular $T^n$-orbit is Hamiltonian $f$-minimal by Proposition \ref{homoge}.
 \end{example}

An {\it almost Calabi-Yau} manifold gives another examples of K\"ahler manifold satisfying \eqref{almE}. A K\"ahler manifold $(M, \omega, J)$ is called  almost Calabi-Yau  if there exists a non-vanishing holomorphic volume form $\Omega$ in the sense of Joyce \cite{Joyce}. In this case, we define a weight function $f$ by
\begin{align}\label{holcf}
e^{2nf}\frac{\omega^n}{n!}=(-1)^{\frac{n(n-1)}{2}}\Big(\frac{\sqrt{-1}}{2}\Big)^n\Omega\wedge\overline{\Omega}.
\end{align}
Then, we have $\rho=ndd^cf$, i.e., $C=0$.  The $f$-minimal Lagrangian submanifold $L$ is a calibrated submanifold for the g.c.K. metric $g_f=e^{2f}g$, i.e., $L$ is volume-minimizing w.r.t. $g_f$ in its homology class (see \cite{Beh}). In particular, any $f$-minimal Lagrangian submanifold  is $f$-stable (See also Corollary \ref{le}). We refer to \cite{Yam} for some examples of Hamiltonian $f$-minimal Lagrangian submanifold in an almost Calabi-Yau manifold.

 When $C>0$, or equivalently, $\rho_f>0$, the (Hamiltonian) $f$-stability of Lagrangian submanifold is a non-trivial proprerty.  If $M$ satisfies \eqref{almE} and $\phi: L\rightarrow M$ is $f$-minimal  (i.e., $K=H_f=0$), we have a simple criterion of the Hamiltonian $f$-stability for the potential function $f$.  The following is a generalization of the result in \cite{CLN} and \cite{Oh1}, and the proof is immediate by Proposition \ref{hmin}.

\begin{theorem}\label{hsta}
Suppose $M$ satisfies $\rho=C\omega+ndd^cf$ for some $f\in C^{\infty}(M)$,  $L$ is compact and $\phi: L\rightarrow M$ is $f$-minimal.  
Then, $\phi$ is Hamiltonian $f$-stable if and only if the first nonzero eigenvalue $\lambda_1(\Delta_f)$ of  the weighted Laplacian $\Delta_f$ satisfies 
\begin{align}
\lambda_1(\Delta_f)\geq C. \label{hstab2}
\end{align}
\end{theorem}

\subsection{The generalized mean curvature form}
In this subsection, we adopt the generalized mean curvature form to the corresponding notions described in \cite{LP} and \cite{SmW}. Moreover, we shall show that the generalized mean curvature form is related to a Hamiltonian invariant of Lagrangian submanifold when $M$ satisfies \eqref{almE}.

First, we recall the definition of {\it Maslov form} and {\it Maslov index} according to \cite{CG},  \cite{LP} and \cite{Oh4}.  Let $(M, \omega, J)$ be an almost K\"ahler manifold and $\phi: L\rightarrow M$  a Lagrangian immersion. Since $\phi$ is Lagrangian, the volume element $d\mu(p)$ determines  a unit length element $\Omega_L(p) \in K(M):=\Lambda^{(n,0)}T^*M$  for each point $p\in L$ which is unique up to sign.  Then,  the square of the element defines a unit length section over $L$:
\begin{align*}
\Omega_L^2: L\rightarrow K^2(M):=K(M)\otimes K(M).
\end{align*}
Note that, if $L$ is orientable, then the volume form of $L$ defines a unit length section $\Omega_L$ of $\phi^*K(M)$ and $\Omega_L^2=\Omega_L\otimes \Omega_L$.

Let $\Sigma$ be a compact oriented surface with boundary $\partial\Sigma$ and $w: \Sigma\rightarrow M$ a smooth map with $w(\partial\Sigma)\subset L$. For simplicity, we assume $\partial\Sigma$ is connected, that is $\partial \Sigma\simeq S^1$, and induce the orientation. Then, $w^*K(M)$ is a trivial bundle, and we have a unit length section $\Omega_w$ of $w^*K(M)$. Since $w(\partial\Sigma)\subset L$, there exists a unique function $e^{\sqrt{-1}\eta}: \partial\Sigma\rightarrow S^1$ so that 
\begin{align*}
\Omega_L^2=e^{\sqrt{-1}\eta}\Omega_w^2
\end{align*}
over $\partial\Sigma$. Now, we define the {\it Maslov index} by minus its winding number
\begin{align*}
\mu_L(w):=\frac{-1}{2\pi}\int_{\partial \Sigma} d\eta.
\end{align*}
This definition recovers the classical definition of Maslov index (See \cite{CG}).

In order to compute the Maslov index, we use a unitary connection $\widetilde{\nab}$ on $TM$. We consider the induced connection on $\phi^*K(M)$. In local, $\Omega_L$ defines a local section of $\phi^*K(M)$. We define a local 1-form $\tilde{\xi}^{\wnab}$ by the connection 1-form in the trivialization $\Omega_L$:
\begin{align*}
\wnab\Omega_L=\tilde{\xi}^{\wnab}\otimes\Omega_L.
\end{align*}
Since $\widetilde{\nab}$ is unitary, one easily verifies that $\tilde{\xi}^{\widetilde{\nab} }$ takes values in ${\rm Im}{\mathbb{C}}$ (see \cite{LP}).  Thus, we define a real 1-form $\xi^{\widetilde{\nab} }\in \Omega^1(L)$ by $\tilde\xi^{\widetilde{\nab}}=\sqrt{-1}\xi^{\widetilde{\nab} }$. 
According to \cite{LP}, we call ${\xi}^{\widetilde{\nab} }$ {\it Maslov form} for the unitary connection ${\widetilde{\nab}}$. Also, we obtain a global 1-form $\tilde{\zeta}^{\wnab}$ on $L$ defined by 
\begin{equation*}
\widetilde{\nab} \Omega_L^2=\tilde{\zeta}^{\widetilde{\nab}}\otimes \Omega_L^2.
\end{equation*}
 By definition, we have $\tilde \zeta^{\widetilde{\nab}}=2\tilde{\xi}^{\wnab}$. 

The following description gives another definition of the Maslov form: Define a tensor field on $L$ by
\begin{align}
\widetilde{S}(X,Y,Z):&=g(J(\wnab_XY)^{\perp}, Z)\label{wS}
\end{align}
for $X,Y,Z\in \Gamma(TL)$. Since $\wnab$ is a unitary connection, one can easily check that $\widetilde{S}$ is symmetric for the last two components. Then, the Maslov 1-form $\xi^{\wnab}$ satisfies (see Section 4.2 in \cite{LP})
\begin{align}\label{Maslov}
\xi^{\wnab}=\sum_{i=1}^n\widetilde{S}(\cdot, e_i, e_i),
\end{align}
where $\{e_i\}_{i=1}^n$ is an orthonormal basis of $L$.

Denote the curvature tensor of $\wnab$ by $\widetilde{R}$, and define a $2$-form on $M$ by
\begin{align*}
\widetilde{P}(X,Y):=\frac{1}{2}{\rm tr}\{J\widetilde{R}(X,Y)\}=\frac{1}{2}\sum_{i=1}^{2n}g(J\widetilde{R}(X,Y)\overline{e}_i, \overline{e}_i),
\end{align*}
where $\{\overline{e}_i\}_{i=1}^{2n}$ is an orthonormal basis of $M$.
As mentioned in Section 4 in \cite{LP}, $\widetilde{P}$ is a closed form and represents $2\pi c_1(M)$, where $c_1(M)$ is the first Chern class of $(M, J)$. Note that, for a general almost K\"ahler manifold,  $\widetilde{P}$ does {\it not} coincide with the usual Ricci form  defined by
\begin{align*}
\tilde{\rho}(X,Y):=\sum_{i=1}^{2n}g(\widetilde{R}(JX,\overline{e}_i)\overline{e}_i, Y).
\end{align*}
However, we shall call the 2-form $\widetilde{P}$ \textit{Ricci form} for the unitary connection $\wnab$. In fact, when $M$ is K\"ahler and $\wnab$ coincides with the Levi-Civita connection $\Nab$, then $\widetilde{P}=\tilde{\rho}$. The following formula is a generalization of so called  Dazord's formula:
\begin{lemma}[Proposition 4.3 in \cite{LP}]\label{Dazord}
The Maslov form $\xi^{\wnab}$ satisfies $d\xi^{\wnab}=\phi^*\widetilde{P}$. 
\end{lemma}

The following integral formula is a generalization of the result of \cite{CG} and \cite{Ono}.
\begin{proposition}\label{intg}
Let $M$ be an almost K\"ahler manifold and $\phi: L\rightarrow M$ be a Lagrangian immersion. For any unitary connection $\widetilde{\nabla}$, we have
\begin{align*}
\mu_L([w])=\frac{1}{\pi}\int_{\Sigma} w^*\widetilde{P}-\frac{1}{\pi}\int_{\partial\Sigma}(\partial w)^*\xi^{\widetilde{\nab}}.
\end{align*}
\end{proposition}
\begin{proof}
The proof is parallel to \cite{CG}.  Define a 1-form $\xi^{\wnab}_w$ on $\partial\Sigma$ by $\wnab \Omega_w^2=2\sqrt{-1}\xi^{\wnab}_w\otimes \Omega_w^2$. Since $\Omega_L^2=e^{\sqrt{-1}\eta}\Omega_w^2$, we see
\begin{align}\label{xi1}
2\xi^{\wnab}=2\xi^{\wnab}_w+d\eta
\end{align}
over $\partial\Sigma$.
On the other hand, by Lemma \ref{Dazord}, we have $d\xi_{w}^{\wnab}=w^*\widetilde{P}$ and hence, the Stokes theorem implies
\begin{align}\label{xi2}
\int_{\Sigma} w^*\widetilde{P}=\int_{\partial \Sigma} \xi_{w}^{\wnab}.
\end{align}
Therefore, by \eqref{xi1} and \eqref{xi2}, we have
\begin{align*}
\mu_L([w])=\frac{-1}{2\pi}\int_{\partial{\Sigma}}d\eta=\frac{1}{\pi}\int_{\partial \Sigma} \xi^{\wnab}_w-\frac{1}{\pi}\int_{\partial \Sigma}(\partial w)^*\xi^{\wnab}
=\frac{1}{\pi}\int_{ \Sigma} \widetilde{P}-\frac{1}{\pi}\int_{\partial \Sigma}(\partial w)^*\xi^{\wnab}.
\end{align*}
This proves the formula.
\end{proof}

Now, we assume $M$ is a K\"ahler manifold. Then, the tensor field $\widetilde{S}$ defined by \eqref{wS} for the Levi-Civita connection $\Nab$ is all symmetric and has the same information of the second fundamental form of $\phi$.  Moreover, the Maslov form for the Levi-Civita connection $\Nab$ coincides with a  mean curvature form $\alpha_H$. If furthermore the Ricci form has a weight function $f$ satisfying \eqref{almE}, then 
\begin{align*}
\alpha_K=\xi^{\wnab}
\end{align*}
 for a unitary connection defined by $\wnab:=\Nab+d^cf\otimes J$ (See \cite{SmW}). Moreover, we see $\widetilde{P}=\rho_f=C\omega$. On the other hand,  Dazord's formula for Lagrangian submanifold in a K\"ahler manifold, i.e., $d\alpha_H=\phi^*\rho$ implies
\begin{align*}
d\alpha_K=d\alpha_H-n\phi^*dd^cf=\phi^*\rho-n\phi^*dd^cf=C\phi^*\omega=0.
\end{align*}
Namely, $\alpha_K$ is always a closed form. In particular, $\alpha_K$ defines a cohomology class $[\alpha_K]\in H^1(L, \mathbb{R})$.

Recall that a Lagrangian submanifold $L$ in a symplectic manifold $(M, \omega)$ is called {\it monotone} if  $L$ satisfies
\begin{align*}
\mu_L([w])={C'}\int_{\Sigma} w^*\omega
\end{align*}
for any $(w,\partial w): (\Sigma, \partial\Sigma)\rightarrow (M,L)$ and some positive constant $C'$ which is independent on the choice of $w$ (cf. \cite{Oh4}).  Note that a monotone Lagrangian submanifold exists only when $M$ is a monotone symplectic manifold. The following result is a generalization of the result due to Oh \cite{Oh3}, H.Ono \cite{Ono} and Cieliebak-Goldstein \cite{CG}:
\begin{proposition}\label{monotone}
Let $(M, \omega, J)$ be a K\"ahler manifold satisfying $\rho=C\omega+ndd^cf$ and $L$ a Lagrangian submanifold in $M$. Then, we have the following:
\begin{enumerate}
\item For any Hamiltonian deformation $\phi_s$ of $L$, $[\alpha_{K_s}]$ gives the same cohomology class.
\item Suppose $C>0$. Then, $[\alpha_K]=0$ if and only if $\phi$ is monotone Lagrangian. In particular, any $f$-minimal Lagrangian submanifold is monotone.
\end{enumerate}
\end{proposition}
\begin{proof}
Since $\rho_f=C\omega$ and $\alpha_{V_s}=du_s$ for $u_s\in C^{\infty}(L)$, \eqref{dsaK} becomes
\begin{align*}
\frac{d}{ds}\alpha_{K_s}=-d\delta_f\alpha_{V_s}+C\alpha_{V_s}=d(\Delta_fu_s+C u_s).
\end{align*}
Integrating this equation, we obtain
\begin{align*}
\alpha_{K_s}-\alpha_{K_0}=d\Big(\int_0^s (\Delta_fu_s+C u_s)ds\Big).
\end{align*}
This proves (1). The second assertion follows from Proposition \ref{intg}. 
\end{proof}

%% file: S3_example.tex
\section{Examples: A torus orbit in weighted projective spaces}

As shown in Example \ref{toric}, there exists a $f$-minimal Lagrangian torus orbit in a toric K\"ahler manifold for any $T^n$-invariant function $f$. In this section, we specify the $f$-minimal Lagrangian torus orbit in weighted projective space for a canonical potential function $f$, and prove the Hamiltonian $f$-stability of the orbits.

\subsection{Weighted projective spaces}
Let $\mathbb{C}^{n+1}$ be the complex Euclidean space with the standard K\"ahler structure $(\omega_{st}, J_{st})$. Note that we have
\begin{align}\label{F}
\omega_{st}=dd^c F\quad {\rm with}\quad F(z):=\frac{1}{4}|z|^2.
\end{align}

  Fix ${\bf a}:=(a_1,\ldots, a_{n+1})\in \mathbb{N}^{n+1}$, and we assume the highest common divisor of all $a_j$'s is equal to $1$.  Define a weighted action of $S^1:=\{e^{\sqrt{-1}\theta}\in \mathbb{C}^*; \theta\in \mathbb{R} \}$ on $\mathbb{C}^{n+1}$  by
\begin{align}\label{action}
e^{\sqrt{-1}\theta}\cdot (z^1,\ldots, z^{n+1}):=(e^{\sqrt{-1}a_1\theta}z^1, \ldots, e^{\sqrt{-1}a_{n+1}\theta}z^{n+1}).
\end{align}
This action is Hamiltonian and a moment map $\mu: \mathbb{C}^{n+1}\rightarrow \mathbb{R}$ is given by
\begin{align*}
\mu(z):=-\frac{1}{2}\sum_{i=1}^{n+1}a_i|z^i|^2. 
\end{align*}
For a regular value $c\in \mathbb{R}$, $S^1$ acts on $\mu^{-1}(c)$ and the symplectic quotient space $\mathbb{C}P^n_{\bf a}:=\mu^{-1}(c)/S^1$ is so called the {\it weighted projective space}.  The standard K\"ahler structure on $\mathbb{C}^{n+1}$ induces a {\it canonical} K\"ahler structure $(\omega_c, J_c, g_c)$ on $\mathbb{C}P^n_{\bf a}$ in the sense of Theorem 7.2.3 in \cite{Fut2}, and $\mathbb{C}P^n_{\bf a}$ becomes a toric K\"ahler orbifold.  
The orbifold structure is described as follows (cf. \cite{Abreu}): For $[z]=[z^1,\ldots, z^{n+1}]\in \mathbb{C}P^n_{\bf a}$, the orbifold structure group at $[z]$ is given by $\mathbb{Z}/m\mathbb{Z}$, where $m$ is the highest common divisor of the set  $\{a_j; j=1,\ldots, n+1\ {\rm s.t.}\ z^j\neq 0\}$. In particular, $[z]$ is a smooth point if and only if $m=1$. Thus, if $z^j\neq 0$ for any $j=1,\ldots, n+1$, then $[z]$ is a smooth point since the highest common divisor of all $a_j$ is equal to 1. Moreover, if elements in any subset of $\{a_j\}_{j=1}^n$ are relatively prime (e.g. all $a_j$ are  prime numbers), then $\mathbb{C}P^n_{\bf a}$ is a smooth manifold.

According to \cite{Kaji}, we compute the Ricci form $\rho_c$ of $\mathbb{C}P^n_{\bf a}$ at a smooth point as follows:
Let us denote the inclusion and the natural projection of $\mu^{-1}(c)$ by $\iota_c:\mu^{-1}(c)\rightarrow \mathbb{C}^{n+1}$ and $\pi: \mu^{-1}(c)\rightarrow \mathbb{C}P^n_{\bf a}$, respectively.  
We set $\mu^{-1}(c)^{reg}:=\{z\in\mu^{-1}(c); [z]\ {\rm is\ a\ smooth\ point}\}$. Note that  $\mu^{-1}(c)^{reg}$ is an open and dense subset of $\mu^{-1}(c)$ and $S^1$ acts on $\mu^{-1}(c)^{reg}$ freely. The tangent space of $\mu^{-1}(c)$ at $z\in \mu^{-1}(c)^{reg}$ is decomposed into $T_z\mu^{-1}(c)=E_z\oplus\mathfrak{g}_z$, where $\mathfrak{g}_z$ denotes the vector space generated by the $S^1$-action, and $E_z$ is the orthogonal complement of $\mathfrak{g}_z$ in  $T_z\mu^{-1}(c)$. Note that  $E_z$ is $J_{st}$-invariant subspace and $\pi_*|_{E_z}: E_z\rightarrow T_{[z]}\mathbb{C}P^n_{\bf a}$ is an isomorphism at $z\in \mu^{-1}(c)^{reg}$.

 Denote the $S^1$-orbit through $z$ by $\mathcal{O}_z$. Since any  orbit $\mathcal{O}_z$ through $z\in \mu^{-1}(c)^{reg}$ is diffeomorphic to $S^1$, we can identify the orbit $\mathcal{O}_z$ with $S^1$. Fix a non-zero element $v\in Lie(S^1)$, and set  $\nu:=\tilde{v}^*$ a dual 1-form on $\mathcal{O}_z\simeq S^1$ of the fundamental vector field $\tilde{v}$. Then, the volume of $\mathcal{O}_z$ w.r.t. the induced metric from $g_{st}$ is computed by
\begin{align}\label{ovolume}
{\rm Vol}_g(\mathcal{O}_z)=|\nu|_g(z)\int_{\mathcal{O}_z}\nu.
\end{align}
Here, the integral $\int_{\mathcal{O}_z}\nu$ is independent of the choice of the principal orbit $\mathcal{O}_z$.  Moreover, the function $|\nu|_g$ is $S^1$-invariant, and hence, we define a well-defined function $|\check{\nu}|$ on $\mathbb{C}P^n_{\bf a}=\mu^{-1}(c)/S^1$ so that $|\nu|_g=|\check{\nu}|\circ \pi$ on $\mu^{-1}(c)$.

By Proposition 2.3 in \cite{Kaji}, the Ricci form $\rho_c$ of $\mathbb{C}P^n_{\bf a}$ at a smooth point  satisfies
\begin{align}\label{ricform}
\pi^*\rho_c=d\gamma_c'+\pi^*dd^c\log|\check{\nu}|
\end{align}
since $\mathbb{C}^{n+1}$ is Ricci-flat, where $\gamma_c'$ is a 1-form on $\mu^{-1}(c)$ defined by
\begin{align*}
\begin{cases}
\gamma_c'(\tilde{X})=-\frac{1}{2}{\rm div}_{\mathbb{C}^{n+1}}(J_{st}\tilde{X})\ {\rm for}\ \tilde{X}\in \mathfrak{g}_z&\\
\gamma_c'(Z)=0\ {\rm for}\ Z\in E_z,
\end{cases}
\end{align*}
where ${\rm div}_{\mathbb{C}^{n+1}}$ is the divergence on $\mathbb{C}^{n+1}$. In our setting, we compute $\gamma'_c$ as follows:

\begin{lemma}\label{lem1}
At the point $z\in {\mu}^{-1}(c)^{reg}$, we have
\begin{align*}
\gamma_c'=\Big(-\frac{1}{c}\sum_{i=1}^{n+1}a_i\Big)(\iota_c^*d^cF-\pi^*d^c\check{F}),
\end{align*}
where $\check{F}$ is the induced function on $\mathbb{C}P^n_{\bf a}$ from $F$ defined by \eqref{F}, i.e., $\iota_c^*F=\pi^*\check{F}$.
\end{lemma}

\begin{proof}
It is sufficient to check for a fundamental vector field $\tilde{X}$. By \eqref{action}, a fundamental vector field is given by
\begin{align}\label{fv}
\tilde{X}_z=\frac{\partial}{\partial \theta}\Big|_{z}=\sum_{i=1}^{n+1}\Big(-a_iy^i\frac{\partial}{\partial x^i}+a_ix^i\frac{\partial}{\partial y^i}\Big)
\end{align}
where $z^i=x^i+\sqrt{-1}y^i$, and hence, we have
\begin{align*}
J_{st}\tilde{X}_z=\sum_{i=1}^{n+1}\Big(-a_ix^i\frac{\partial}{\partial x^i}-a_iy^i\frac{\partial}{\partial y^i}\Big).
\end{align*}
Therefore, we see
\begin{align*}
\gamma_c'(\tilde{X}_z)=-\frac{1}{2}{\rm div}_{\mathbb{C}^{n+1}}(J_{st}\tilde{X}_z)=-\frac{1}{2}\sum_{i=1}^{n+1}\Big{\{} \frac{\partial}{\partial x^i}(-a_ix^i)+\frac{\partial}{\partial y^i}(-a_iy^i)\Big{\}}=\sum_{i=1}^{n+1}a_i.
\end{align*}
On the other hand, we see
\begin{align*}
d^cF(\tilde{X}_z)&=-dF(J_{st}\tilde{X}_z)=\frac{-1}{2}\sum_{i=1}^{n+1}(x^idx^i+y^idy^i)\Big{\{}\sum_{i=1}^{n+1}\Big(-a_ix^i\frac{\partial}{\partial x^i}-a_iy^i\frac{\partial}{\partial y^i}\Big) \Big{\}}\\
&=\frac{1}{2}\sum_{i=1}^{n+1}\{a_i(x^i)^2+a_i(y^i)^2\}=\frac{1}{2}\sum_{i=1}^{n+1}a_i|z^i|^2=-\mu(z)=-c
\end{align*}
since we suppose $z\in\mu^{-1}(c)$. Therefore, we obtain
\begin{align*}
\gamma_c'(\tilde{X}_z)=\Big(-\frac{1}{c}\sum_{i=1}^{n+1}a_i\Big)d^cF(\tilde{X}_z)
\end{align*}
This proves the lemma.
\end{proof}

Moreover, the equation \eqref{fv} shows that we can take $\nu$ so that
\begin{align}\label{nu}
|\nu|_g(z)=|\tilde{X}_z|_g=\Big(\sum_{i=1}^{n+1}a_i^2|z^i|^2\Big)^{1/2}.
\end{align}

By Lemma \ref{lem1}, the formula \eqref{ricform} becomes
\begin{align*}
\pi^*\rho_c&=\Big(-\frac{1}{c}\sum_{i=1}^{n+1}a_i\Big)(\iota_c^*dd^cF-\pi^*dd^c\check{F})+\pi^*dd^c\log|\check{\nu}|\\
&=\pi^*\Big{\{}\Big(-\frac{1}{c}\sum_{i=1}^{n+1}a_i\Big)(\omega_c-dd^c\check{F})+dd^c\log|\check{\nu}|\Big{\}}
\end{align*}
Since $\pi$ is surjective, the Ricci form $\rho_c$ of $\mathbb{C}P^{n}_{\bf a}$ at a smooth point is expressed by
\begin{align}
\rho_c&=\Big(-\frac{1}{c}\sum_{i=1}^{n+1}a_i\Big)\omega_c+ndd^cf_{\bf a}\quad {\rm with}\quad f_{\bf a}:=\frac{1}{n}\Big{\{}\Big(\frac{1}{c}\sum_{i=1}^{n+1}a_i\Big)\check{F}+\log|\check{\nu}|\Big{\}}.\label{rhocf}
\end{align}
  We call this function $f_{\bf a}$ the {\it canonical} weight function for $\mathbb{C}P^n_{\bf a}$. Note that $f_{\bf a}$ is constant on $\mathbb{C}P^n_{\bf a}$ if and only if $\pi^*f_{\bf a}=\iota_c^*(\kappa F+\log|\nu|_g)$ is constant on $\mu^{-1}(c)$, where $\kappa=(1/c)\sum_{i=1}^{n+1}a_i$. This is actually the case when $a_1=\ldots=a_{n+1}$ and the $S^1$-action is the standard Hopf action. 

\subsection{Clifford torus}
Let us consider the standard $T^{n+1}$-action on $\mathbb{C}^{n+1}$ defined by
\begin{align*}
(e^{\sqrt{-1}\theta_1},\ldots,e^{\sqrt{-1}\theta_{n+1}})\cdot(z^1,\ldots, z^{n+1}):=(e^{\sqrt{-1}\theta_1}z^1,\ldots, e^{\sqrt{-1}\theta_{n+1}}z^{n+1}).
\end{align*}
and the {\it Clifford torus} $T^{n+1}_{r}$of radius $r$:
\begin{align*}
T^{n+1}_r:=T^{n+1}\cdot (r,\ldots, r).
\end{align*}
$T^{n+1}_{r}$ is Lagrangian in $\mathbb{C}^{n+1}$ and invariant under the weighted $S^1$-action. Thus, $T^{n+1}_{r}$ is contained in the level set $\mu^{-1}(c)$ of the $S^1$-action with 
\begin{align}\label{c}
c=-\frac{r^2}{2}\sum_{i=1}^{n+1}a_i.
\end{align}

The orbit space $T^n_{\bf a}:=T^{n+1}_r/S^1$ is a Lagrangian submanifold without any singular point in the symplectic quotient space $\mathbb{C}P^n_{\bf a}$. Denote the mean curvature vector of $\phi_c: T^n_{\bf a}\rightarrow \mathbb{C}P^{n}_{\bf a}$ by $H_c$, and set $\beta_{H_c}:=\phi^*_c(i_{H_c}\omega_c)$.  Then, we obtain the following:
\begin{proposition}\label{lem3} 
The mean curvature form of $\phi_c: T^n_{\bf a}\rightarrow \mathbb{C}P^{n}_{\bf a}$ is given by
\begin{align*}
\beta_{H_c}=n\phi^*_cd^cf_{\bf a}.
\end{align*}
In particular, $T^n_{\bf a}$ is $f_{\bf a}$-minimal in $\mathbb{C}P^{n}_{\bf a}$ for the canonical potential function $f_{\bf a}$ defined by \eqref{rhocf}.
\end{proposition}

\begin{proof}
First, we note that  the mean curvature vector $H$ of the Clifford torus $T^{n+1}_r\rightarrow \mathbb{C}^{n+1}$ is given by
\begin{align}\label{H1}
H=-\frac{2}{r^2}\Nab F. 
\end{align}
Denote the mean curvature vector  of the immersion $T^{n+1}_r\rightarrow \mu^{-1}(c)$ by $H'$.   Then, by Proposition 2.5 in \cite{Kaji}, we have
\begin{align}\label{H2}
\pi^*\beta_{H_c}=\alpha_{H'}+\pi^*\phi_c^*(d^c\log|\check{\nu}|)
\end{align}

We shall compute $\alpha_{H'}$.
The tangent space of $T^{n+1}_r$ is decomposed into $T_pT^{n+1}_r=E_p^l\oplus \mathfrak{g}_p$, where $E_p=E_p^l\oplus JE_p^l$. For the natural projection $\pi: T^{n+1}_r\rightarrow T^n_{\bf a}$, $\pi_*|_{E_p^l}:E_p^l\rightarrow T_{\pi(p)}T^n_{\bf a}$ is an isomorphism. We denote the Levi-Civita connections on $\mathbb{C}^{n+1}$, $\mu^{-1}(c)$ and the quotient space $\mu^{-1}(c)/S^1$ by $\Nab$, $\Nab'$ and $\Nab_c$, respectively (See \cite{Kaji} for a general description).

Take arbitrary $Z\in E_p^l$. Note that $JZ\in E_p\subset T_p\mu^{-1}(c)$. Thus, by using \eqref{H1},  we see
\begin{align*}
\alpha_{H'}(Z)_p&=g(JH'_z, Z)_p=-g(H, JZ)_p\\
&=\frac{2}{r^2}g(\Nab F, JZ)_p=\frac{2}{r^2}g(\Nab'F, JZ)_p\\
&=\frac{2}{r^2}g_c(\pi_*\Nab'F, \pi_*JZ)_{\pi(p)}=\frac{2}{r^2}g_c(\Nab_c\check{F}, J_c\pi_*Z)_{\pi(p)}\\
&=-\frac{2}{r^2}\pi^*\phi^*_cd^c\check{F}_{\pi(p)}(Z).
\end{align*}
Moreover, we have $\alpha_{H'}(\tilde{X})=-g(H', J\tilde{X})=0$ for any $\tilde{X}\in \mathfrak{g}_p$. Therefore, we obtain $\alpha_{H'}=\pi^*\phi^*_cd^c(-\frac{2}{r^2}\check{F})$.

Recall that $f_{\bf a}$ is given by
\begin{align*}
f_{\bf a}=\frac{1}{n}\Big\{\Big(\frac{1}{c}\sum_{i=1}^{n+1}a_i\Big)\check{F}+\log|\check{\nu}|\Big{\}}=\frac{1}{n}\Big(-\frac{2}{r^2}\check{F}+\log|\check{\nu}|\Big)
\end{align*}
by \eqref{rhocf} and \eqref{c}.
Therefore, by \eqref{H2}, we see
\begin{align*}
\pi^*\beta_{H_c}=\pi^*\phi^*_cd^c\Big(-\frac{2}{r^2}\check{F}\Big)+\pi^*\phi_c^*(d^c\log|\check{\nu}|)=n\pi^*\phi_c^*d^cf_{\bf a}
\end{align*}
Because $\pi$ is surjective,  this proves $\beta_{H_c}=n\phi^*_cd^cf_{\bf a}$.
\end{proof}

Since $T^{n}_{\bf a}=T^{n+1}_r/S^1$ is homogeneous and $F$ and $|\nu|$ are $T^{n+1}$-invariant functions (see \eqref{F} and \eqref{nu}), the weight function $f_{\bf a}$ is constant on $T^{n}_{\bf a}$ and the weighted Laplacian $\Delta_{f_{\bf a}}$ coincides with the usual Laplacian $\Delta$ w.r.t. the  metric $\phi^*_cg_c$. Note that $\phi^*_cg_c$ coincides with the induced flat metric on the torus $T^{n}_{\bf a}$ from the flat metric on $T^{n+1}_r$ by the construction. Therefore, by Theorem \ref{hsta}, \eqref{rhocf} and \eqref{c},  the Hamiltonian $f_{\bf a}$-stability of $T^{n}_{\bf a}$  is equivalent to
\begin{align*}
\lambda_1\geq-\frac{1}{c}\sum_{i=1}^{n+1}a_i=\frac{2}{r^2},
\end{align*}
where $\lambda_1$ is the first eigenvalue of the usual Laplacian $\Delta$. 
In the following, we restrict our attention to the case when ${\bf a}:=(1,a_2,\ldots, a_{n+1})$ for simplicity. Then, we have the following result:

\begin{theorem}\label{clifford}
Let $T^n_{\bf a}=T^{n+1}_{r}/S^1$ be the flat torus obtained by the quotient of the weighted $S^1$-action for ${\bf a}:=(1,a_2,\ldots, a_{n+1})$ on the Clifford torus $T^{n+1}_{r}$ of radius $r$.  Then, the first eigenvalue of the Laplacian $\Delta$ acting on $C^{\infty}(T^n_{\bf a})$ w.r.t. the flat metric on $T^{n}_{\bf a}$ satisfies
\begin{align*}
\lambda_1\geq \frac{2}{r^2}.
\end{align*}
Here, the equality holds if and only if at least two components of ${\bf a}=(1,a_2,\ldots,a_{n+1})$ are equal. In particular, $T^n_{\bf a}$ is $f_{\bf a}$-minimal and Hamiltonian $f_{\bf a}$-stable in $\mathbb{C}P^{n}_{\bf a}$  for the canonical weight function $f_{\bf a}$ defined by \eqref{rhocf}.
\end{theorem}

The rest of this section is devoted to give a proof of Theorem \ref{clifford}.  

We identify the flat torus $T^{n+1}_{r}$ with $\mathbb{R}^{n+1}/\Sigma$, where 
\begin{align*}
\Sigma:={\rm span}_{\mathbb{Z}}\Big{\{}{}^t\!(R,0,\ldots, 0),{}^t\!(0,R,0,\ldots, 0)\cdots, {}^t\!(0,\ldots,0,R)\Big{\}}\subset \R^{n+1}, \quad R:=2\pi r
\end{align*}
For ${x}={}^t\!({x}^1,\ldots, {x}^{n+1})\in\R^{n+1}$, we denote $[x]$ the equivalent class of $x$ in $\mathbb{R}^{n+1}/\Sigma$, namely, 
\begin{align*}
[x]=\{x+R{\bf m}; {\bf m}={}^t\!(m_1,\ldots, m_{n+1})\in\Z^{n+1}\}.
\end{align*}
 Then, 
an $S^1$-orbit $\widetilde{\mathcal{O}}_{[x]}$ through $[x]\in \R^{n+1}/\Sigma$ is given by
\begin{align*}
\widetilde{\mathcal{O}}_{[x]}={\mathcal{O}}_{x}/\Sigma, \quad \mathcal{O}_{{x}}:=\Big{\{}{}^t\!({x}^1+\theta,{x}^2+a_2\theta, \ldots, {x}^{n+1}+a_{n+1}\theta);  \theta\in \mathbb{R}\Big{\}}.
\end{align*}
Notice that any orbit intersects to the hypersurface
\begin{align*}
\tilde{\Pi}_{\bf a}:=\Pi_{\bf a}/\Sigma,\quad \Pi_{\bf a}:=\Big{\{}{}^t\!(x^1,\ldots, x^{n+1})\in \mathbb{R}^{n+1}; x^1+\sum_{i=2}^{n+1}a_ix^i=0\Big{\}}\ni O.
\end{align*}
where $O$ is the origin of $\R^{n+1}$.
Denote the intersection point of ${\mathcal{O}}_{{x}}$ and $\Pi_{\bf a}$ by $P({x})$. More precisely, 
\begin{align}
P({x})&=\overrightarrow{OP}({x})={}^t\!\Big({x}^1+\theta({x}), {x}^2+a_2\theta({x}),\ldots, {x}^{n+1}+a_{n+1}\theta({x})\Big)\quad {\rm with}\label{opx}\\
\theta({x})&:=\frac{1}{1+\sum_{i=2}^{n+1}a_i^2}\Big(-{x}^1-\sum_{i=2}^{n+1}a_i{x}^i\Big).\nonumber
\end{align}
Note that $P: \R^{n+1}\rightarrow \Pi_{\bf a}$ is a linear map.

We set
\begin{align*}
a_1:&=0, \quad
\begin{cases}
X_0:=0,\\
X_s:={\displaystyle \sum_{i=n+2-s}^{n+1}a_i}\quad ({\rm for}\ s=1,\ldots,n) \\
X_{n+1}:=X_n
\end{cases},
\quad S:=1+\sum_{i=2}^{n+1}a_i^2, \\
N_s:&=a_{n+1-s}-a_{n+2-s}=X_{s-1}-2X_s+X_{s+1}\quad ({\rm for}\ s=1,\ldots,n).
\end{align*}
Moreover, we define vectors in $\mathbb{R}^{n+1}$ by
\begin{align*}
{\bf m}_s:={}^t\!(\underbrace{0,\ldots, 0}_{n+1-s}, \underbrace{1,\ldots, 1}_{s})\quad ({\rm for}\ s=1,\ldots, n),
\end{align*}
and, denote the intersection point of the orbit $\mathcal{O}_{R{\bf m}_s}$ and $\Pi_{\bf a}$ by $A_{s}:=P(R{\bf m}_s)$. More precisely, 
\begin{align}\label{oa}
A_s&=\overrightarrow{OA_s}={}^t\!(x_{s}^1, \ldots, x_{s}^{n+1})\ {\rm with}
\begin{cases}
x_s^1=\frac{R}{S}(-X_s)\\
x_s^k=\frac{R}{S}(-a_kX_s)\quad{\rm for}\ k=2,\ldots, n+1-s\\
x_s^l=\frac{R}{S}(S-a_lX_s)\quad{\rm for}\ l=n+2-s,\ldots, n+1
\end{cases}.
\end{align}
Also, we define a symmetric matrix $A$ by
\begin{align*}
A:=\{\langle \overrightarrow{OA_s},\overrightarrow{OA_t}\rangle\}_{1\leq s, t\leq n}
\end{align*}

\begin{proposition}\label{cl1}
The vectors $\{\overrightarrow{OA_s}\}_{1\leq s\leq n}$ defines a lattice on the hyperplane $\Pi_{\bf a}$:
\begin{align*}
\Gamma_{\bf a}:={\rm span}_{\mathbb{Z}}\{\overrightarrow{OA_s}\}_{1\leq s\leq n}
\end{align*} 
Moreover, the orbit space $T^{n+1}_{1/\sqrt{n+1}}/S^1$ for the weighted $S^1$-action of ${\bf a}=(1,a_2,\ldots,a_{n+1})$ is given by  $\Pi_{\bf a}/\Gamma_{\bf a}$.
\end{proposition}

In order to prove this proposition, we need the following lemma:

\begin{lemma}\label{cl2}
The matrix $A$ is invertible and the inverse matrix is given by a symmetric matrix
\begin{align}
A^{-1}&=\frac{1}{R^2}\{B_{st}\}_{1\leq s, t\leq n},\quad {\rm where}\nonumber\\
&B_{st}= 
\begin{cases}
    \ \ 2+N_s^2 & (t=s)\\
    -1+N_sN_{s+1} & (t=s+1)\\
        N_sN_t & (s+2\leq t\leq n)\\
  \end{cases}
  \quad {\rm for}\ s=1,\ldots, n-1\quad {\rm and}\label{bmtrx1}\\
&B_{nn}=1+N_n^2.\label{bmtrx2}
\end{align} 
In particular, $\{\overrightarrow{OA_s}\}_{1\leq s\leq n}$ gives a basis of $\Pi_{\bf a}$, and we have
\begin{align}\label{opv}
\overrightarrow{OP}({x})=\sum_{s=1}^n\frac{1}{R}\Big{\{}N_s{x}^1-({x}^{n+1-s}-{x}^{n+2-s})\Big{\}}\overrightarrow{OA_s}.
\end{align}
\end{lemma}

\begin{proof}
Suppose $1\leq s\leq t\leq n$. Then, by \eqref{oa},  we compute
\begin{align*}
A_{st}:&=\langle \overrightarrow{OA_s},\overrightarrow{OA_t}\rangle\\
&=x_s^1x_t^1+\sum_{k=2}^{n+1-t} x_s^kx_t^k+\sum_{l=n+2-t}^{n+1-s}x_s^lx_t^l+\sum_{j=n+2-s}^{n+1}x_s^jx_t^j\\
&=\frac{R^2}{S^2}X_sX_t+\frac{R^2}{S^2}\sum_{k=2}^{n+1-t}(-a_kX_s)(-a_kX_t)+\frac{R^2}{S^2}\sum_{l=n+2-t}^{n+1-s}(-a_lX_s)(S-a_lX_t)\\
&\quad+\frac{R^2}{S^2}\sum_{j=n+2-s}^{n+1}(S-a_jX_s)(S-a_jX_t)\\
&=\frac{R^2}{S^2}\Big{\{}SX_sX_t-S\Big(\sum_{l=n+2-t}^{n+1-s}a_l\Big)X_s-S\Big(\sum_{j=n+2-s}^{n+1}a_j\Big)X_s-S\Big(\sum_{j=n+2-s}^{n+1}a_j\Big)X_t+sS^2\Big{\}}\\
&=\frac{R^2}{S}(-X_sX_t+sS).
\end{align*}
Note that this expression depends on the ordering $s\leq t$.

First, we consider the case when $1\leq s\leq u<n$. Then, by \eqref{bmtrx1} and \eqref{bmtrx2}, we have
\begin{align}
\frac{1}{R^2}\sum_{t=1}^nA_{st}B_{tu}&=\frac{1}{R^2}\Big{\{}\sum_{t=1}^{u-2}A_{st}N_tN_u+A_{su-1}(-1+N_{u-1}N_u)+A_{su}(2+N_u^2)+A_{su+1}(-1+N_{u+1}N_u)\nonumber\\ 
&\quad +\sum_{t=u+2}^{n}A_{st}N_uN_u\Big{\}}\nonumber\\ 
&=\frac{1}{R^2}\Big{\{}\sum_{t=1}^{n}A_{st}N_tN_u+(-A_{su-1}+2A_{su}-A_{su+1})\Big{\}}\nonumber\\
&=\frac{1}{S}\Big{[}\sum_{t=1}^{s}(-X_sX_t+tS)N_tN_u+\sum_{t=s+1}^{n}(-X_sX_t+sS)N_tN_u\nonumber\\ 
&\quad+\{-(-X_sX_{u-1}+s'S)+2(-X_sX_u+sS)-(-X_sX_{u+1}+sS)\}\Big{]}\nonumber\\
&=\frac{1}{S}\Big{[}-X_sN_u\sum_{t=1}^nX_tN_t+SN_u\Big(\sum_{t=1}^stN_t+s\sum_{t=s+1}^nN_t\Big)+X_sN_u+(s-s')S\Big{]}, \label{inveq1}
\end{align}
where we set
\begin{align*}
s':=
\begin{cases}
s& {\rm if}\ s<u\\
u-1& {\rm if}\ s=u
\end{cases}.
\end{align*}
Here, we compute the coefficient of the first term in \eqref{inveq1} as follows:
\begin{align*}
\sum_{t=1}^{n}X_tN_t&=\sum_{t=1}^nX_t(a_{n+1-t}-a_{n+2-t})\\
&=\sum_{t=1}^n\{X_ta_{n+1-t}-(X_{t-1}-a_{n+2-t})a_{n+2-t}\}\\
&=-X_0a_{n+1}+X_na_1-\sum_{t=1}^na_{n+2-t}^2=1-S
\end{align*}
since $X_t=X_{t-1}+a_{n+2-t}$ and $X_0=a_1=0$. 
Moreover, the coefficient of the second term in \eqref{inveq1} is computed by
\begin{align*}
\sum_{t=1}^{s}tN_t+s\sum_{t=s+1}^{n}N_t&=\{(s-1)X_s-2sX_s+sX_{s+1}\}+s(-a_{n+1-s}+a_1)\\
&=-X_s+sa_{n+1-s}+s(-a_{n+1-s}+a_1)=-X_s.
\end{align*}
 Therefore, \eqref{inveq1} becomes 
\begin{align*}
\frac{1}{R^2}\sum_{t=1}^nA_{st}B_{tu}&=\frac{1}{S}\Big{[}-X_sN_u(1-S)+SN_u(-X_s)+X_sN_u+(s-s')S\Big{]}\\
&=s-s'=
\begin{cases}
0& {\rm if}\ s<u\\
1& {\rm if}\ s=u
\end{cases}.
\end{align*}

Similarly, when $1\leq s< u=n$,  we see
\begin{align*}
\frac{1}{R^2}\sum_{t=1}^nA_{st}B_{tn}&=\frac{1}{R^2}\Big{\{}\sum_{t=1}^{n-2}A_{st}N_tN_n+A_{sn-1}(-1+N_{n-1}N_n)+A_{sn}(1+N_n^2)\Big{\}}\\
&=\frac{1}{R^2}\Big{\{}\sum_{t=1}^{n}A_{st}N_tN_n+(-A_{sn-1}+A_{sn})\Big{\}}\\
&=\frac{1}{S}\Big{[}-X_sN_n\sum_{t=1}^nX_tN_t+SN_n\Big(\sum_{t=1}^stN_t+s\sum_{t=s+1}^nN_t\Big)+X_s(X_{n-1}-X_n)+(s-s')S\Big{]}\\ 
&=\frac{1}{S}\Big{[}-X_sN_n(1-S)+SN_n(-X_s)+X_sN_n+(s-s')S\Big{]}\\
&=s-s'=
\begin{cases}
0& {\rm if}\ s<n\\
1& {\rm if}\ s=n
\end{cases}.
\end{align*}
This proves $A^{-1}=\frac{1}{R^2}\{B_{st}\}_{s,t}$ is the inverse matrix of $A$. In particular, $\{\overrightarrow{OA_s}\}_{1\leq s\leq n}$ is a basis of the hyperplane $\Pi_{\bf a}$.

For any $x\in \mathbb{R}^{n+1}$, we set
\begin{align*}
\overrightarrow{OP}({x}):=\sum_{s=1}^n\alpha_s(x) \overrightarrow{OA_s}.
\end{align*}
Then, we have
\begin{align*}
\alpha_s(x)=\frac{1}{R^2}\sum_{t=1}^nB_{st}\langle\overrightarrow{OP}({x}), \overrightarrow{OA_t}\rangle.
\end{align*}
Here, by \eqref{opx} and \eqref{oa}, we compute
\begin{align*}
\langle\overrightarrow{OP}({x}), \overrightarrow{OA_t}\rangle
&=\frac{R}{S}\Big{[}\{{x}^1+\theta(x)\}(-X_t)+\sum_{k=2}^{n+1-t}\{{x}^k+a_k\theta(x)\}(-a_kX_t)+\sum_{l=n+2-t}^{n+1}\{{x}^l+a_l\theta(x)\}(S-a_lX_t)\Big{]}\\
&=\frac{R}{S}\Big{[}-X_t\Big({x}^1+\sum_{i=2}^{n+1}a_i{x}^i\Big)-\theta(x)X_tS+\sum_{l=n+2-t}^{n+1}\{{x}^l+a_l\theta(x)\}S\Big{]}\\
&=\frac{R}{S}\Big{[}-X_t\Big({x}^1+\sum_{i=2}^{n+1}a_i{x}^i\Big)+S\Big(\sum_{l=n+2-t}^{n+1}{x}^l\Big)\Big{]}\\
&=\frac{R}{S}(-X_tT+Y_tS),
\end{align*}
where we set  $T:={x}^1+\sum_{i=2}^{n+1}a_i{x}^i$ and $Y_t:=\sum_{l=n+2-t}^{n+1}{x}^l$. 
Thus, by \eqref{bmtrx1} and \eqref{bmtrx2}, we see
\begin{align*}
\alpha_s(x)&=\frac{1}{R^2}\sum_{t=1}^nB_{st}\frac{R}{S}(-X_tT+Y_tS)\\
&=\frac{1}{RS}\Big{\{}\sum_{t=1}^nN_sN_t(-X_tT+Y_tS)-(-X_{s-1}T+Y_{s-1}S)+2(-X_{s}T+Y_{s}S)-(-X_{s+1}T+Y_{s+1}S)\Big{\}}\\
&=\frac{1}{RS}\Big{\{}-N_sT\sum_{t=1}^nX_tN_t+N_sS\sum_{t=1}^nY_tN_t+N_sT-(Y_{s-1}-2Y_s+Y_{s+1})S\Big{\}}\\
&=\frac{1}{RS}\Big{\{}-N_sT(1-S)+N_sS({x}^1-T)+N_sT-({x}^{n+1-s}-{x}^{n+2-s})S\Big{\}}\\
&=\frac{1}{R}\Big{\{}N_s{x}^1-({x}^{n+1-s}-{x}^{n+2-s})\Big{\}}.
\end{align*}
This proves the lemma.
\end{proof}

Now, we give a proof of Proposition \ref{cl1}.

\begin{proof}[Proof of Proposition \ref{cl1}]
Let us denote the fundamental domain of the lattice $\Gamma_{\bf a}$ by
\begin{align*}
D:=\Big{\{}\sum_{s=1}^nt_s\overrightarrow{OA_s}; 0\leq t_s\leq 1\Big{\}}\subset \Pi_{\bf a}\subset \R^{n+1}.
\end{align*}

Take an arbitrary $S^1$-orbit $\widetilde{\mathcal{O}}_{[{x}]}=\mathcal{O}_x/\Sigma$ through $[x]\in\R^{n+1}/\Sigma\simeq T^{n+1}_{1/\sqrt{n+1}}$. Since any $\mathcal{O}_{x}$ intersects to $\Pi_{\bf a}$ at $P(x)$ and $\Gamma_{\bf a}$ defines a lattice on $\Pi_{\bf a}$,  there exists $x_D\in D$ so that $\mathcal{O}_{x_D}/\Sigma=\widetilde{\mathcal{O}}_{[{x}]}$.  We shall show such $x_D$ is uniquely determined up to $\Gamma_{\bf a}$-action. Then, we obtain a well-defined map $\widetilde{\mathcal{O}}_{[x]}\mapsto x_D\in \Pi_{\bf a}/\Gamma_{\bf a}$. 

Suppose an element $x'_D\in D$ satisfies  $\mathcal{O}_{x'_D}/\Sigma=\mathcal{O}_{x_D}/\Sigma=\widetilde{\mathcal{O}}_{[{x}]}$.  Then, we may find $x'\in \mathcal{O}_{x'_D}$ so that $x'=x_D+R{\bf m}$ for some ${\bf m}:={}^t\!(m_1,\ldots,m_{n+1})\in \mathbb{Z}^{n+1}$. Moreover, we see 
\begin{align*}
x'_D=P(x')=P({x}_D+R{\bf m})=x_D+P(R{\bf m})
\end{align*}
since $P: \R^{n+1}\rightarrow \Pi_{\bf a}$ is a linear operator. Here, by \eqref{opv}, we have 
\begin{align*}
P(R{\bf m})=\sum_{s=1}^n\Big{\{}M_sm_1-(m_{n+1-s}-m_{n+2-s})\Big{\}}\overrightarrow{OA_s}.
\end{align*}
In particular, $P(R{\bf m})\in \Gamma_{\bf a}$ since $M_s$ and $m_i$ are integers. Therefore,  $x'_D$ coincides with $x_D$ up to $\Gamma_{\bf a}$-action.

By construction, the map $\widetilde{\mathcal{O}}_{[x]}\mapsto x_D\in \Pi_{\bf a}/\Gamma_{\bf a}$ is bijective, and hence, the orbit space $T^{n+1}_{1/\sqrt{n+1}}/S^1$ is identified with $ \Pi_{\bf a}/\Gamma_{\bf a}$. This proves the proposition.
\end{proof}

Let us consider the dual lattice $\Gamma^*_{\bf a}$ defined by
\begin{align*}
\Gamma^*_{\bf a}:=\{y\in \Pi_{\bf a}: \langle y, x\rangle\in\mathbb{Z}\ \forall x\in\Gamma_{\bf a}\}=\{y\in \Pi_{\bf a}: \langle y, \overrightarrow{OA_i}\rangle\in\mathbb{Z}\ \forall i=1,\ldots, n\}.
\end{align*}

It is easy to see that $\Gamma^*_{\bf a}$ is spanned by the metric dual $\{\overrightarrow{OA_s^*}\}_{s=1,\ldots,n}$ of $\{\overrightarrow{OA_s}\}_{s=1,\ldots, n}$ on $\Pi_{\bf a}$, namely,  $\overrightarrow{OA_s^*}$ is defined by the relation 
\begin{align*}
\langle \overrightarrow{OA_s^*}, \overrightarrow{OA_t}\rangle=\delta_{st}
\end{align*}
for any $s,t=1,\ldots, n$, or more precisely, 
\begin{align*}
\overrightarrow{OA_s^*}:=\frac{1}{R^2}\sum_{t=1}^n B_{st}\overrightarrow{OA_t}.
\end{align*} 
In particular, we have
\begin{align*}
\langle\overrightarrow{OA_s^*},\overrightarrow{OA_t^*} \rangle=\frac{1}{R^2}B_{st}.
\end{align*}

It is known that the first eigenvalue of $\lambda_1$ of the Laplacian $\Delta$ on the flat torus $\Pi_{\bf a}/\Gamma_{\bf a}$ is given by
\begin{align*}
\lambda_1=4\pi^2d^2,
\end{align*}
where $d$ is the shortest distance between $0$ and $\Gamma^*_{\bf a}\setminus \{0\}$ (See Lemma 5.5 in \cite{Oh1}).
We shall estimate the shortest distance, namely, we consider
\begin{align*}
d^2:={\rm min}\{|y|^2; y\in \Gamma^*_{\bf a}\setminus\{0\}\}.
\end{align*}
Since $\Gamma^*_{\bf a}$ is a lattice, it is sufficient to consider the following element:
\begin{align*}
y\in\Big{\{}\sum_{s=1}^n \delta_s\overrightarrow{OA_s^*}; \delta_s=0\ {\rm or}\ \pm1,  \forall s=1,\ldots, n\Big{\}}\setminus \{0\}.\end{align*}
Then, by \eqref{bmtrx1} and \eqref{bmtrx2}, we have
\begin{align}
|y|^2&=\frac{1}{R^2}\sum_{s,t=1}^n \delta_s\delta_tB_{st}=\frac{1}{R^2}\sum_{s=1}^{n-1}\Big(\delta_s^2B_{ss}+2\delta_s\delta_{s+1}B_{ss+1}+2\sum_{t=s+2}^n\delta_s\delta_tB_{st}\Big)+\frac{\delta_n^2}{R^2}B_{nn}\nonumber\\
&=\frac{1}{R^2}\sum_{s=1}^{n-1}\Big{\{}\delta_s^2(2+N_s^2)+2\delta_s\delta_{s+1}(-1+N_sN_{s+1})+2\sum_{t=s+2}^n\delta_s\delta_tN_sN_t\Big{\}}+\frac{\delta_n^2}{R^2}(1+N_n^2)\nonumber\\
&=\frac{2}{R^2}\sum_{s=1}^{n-1}(\delta_s^2-\delta_s\delta_{s+1})+\frac{\delta_n^2}{R^2}+\frac{1}{R^2}\Big(\sum_{s=1}^{n}\delta_sN_s\Big)^2. \label{y2}
\end{align}

Since we assume $\delta_s=0$ or $\pm1$, we easily see $\delta_s^2-\delta_s\delta_{s+1}$ is an non-negative integer and  $\delta_s^2-\delta_s\delta_{s+1}=0$ iff $\delta_s=0$ or $\delta_s=\delta_{s+1}$. We divide into three cases for combinations of $\delta_s$.
\begin{itemize}
\item The case when $\delta_{s}^2-\delta_{s}\delta_{s+1}\neq 0$ for some $s\in \{1,\ldots, n-1\}$. Let $s=s_0$ be such an element.  Then,  $\delta_{s_0}^2-\delta_{s_0}\delta_{s_0+1}\geq 1$ and \eqref{y2} implies
\begin{align*}
|y|^2\geq \frac{2}{R^2}(\delta_{s_0}^2-\delta_{s_0}\delta_{s_0+1})\geq \frac{2}{R^2}.
\end{align*}
 The equality holds in both inequalities iff 
\begin{itemize}
\item $\delta_n=\sum_{s=1}^n \delta_sN_s=0$, and
\item there exists a unique element $s_0\in \{1,\ldots, n-1\}$ s.t. $\delta_{s_0}^2-\delta_{s_0}\delta_{s_0+1}=1$ and $\delta_{s}^2-\delta_{s}\delta_{s+1}=0$ for $s\in\{1,\ldots, n-1\}\setminus\{s_0\}$.
\end{itemize} 
If this is the case, then there exists a number $t_0\leq s_0$ so that 
\begin{align*}
\begin{cases}
\delta_s=0\quad {\rm for}\quad 1\leq s\leq t_0-1\quad {\rm and}\quad s_0+1\leq s\leq n, \\
 \delta_s=\delta_{s_0}\neq 0\quad {\rm for}\quad t_0\leq s\leq s_0,
\end{cases}
\end{align*}
and 
\begin{align*}
0=\sum_{s=1}^n \delta_sN_s=\delta_{s_0}\sum_{s=t_0}^{s_0}(a_{n+1-s}-a_{n+2-s})=\delta_{s_0}(-a_{n+2-t_0}+a_{n+1-s_0}),
\end{align*}
namely, we obtain $a_{n+1-s_0}=a_{n+2-t_0}$. Conversely, if $a_{n+1-s_0}=a_{n+2-t_0}$ for some $t_0\leq s_0$, then we have $|\sum_{s=t_0}^{s_0} \overrightarrow{OA_s^*}|^2={2}/{R^2}$.
\item The case when $\delta_s^2-\delta_s\delta_{s+1}=0$ for all $s\in \{1,\ldots, n-1\}$ and $\delta_s=0$ for all $s\in \{1,\ldots, n-1\}$. Then $\delta_n=\pm 1$ and \eqref{y2} implies
\begin{align*}
|y|^2=\frac{1}{R^2}+\frac{1}{R^2}N_n^2=\frac{1}{R^2}(1+a_2^2)\geq \frac{2}{R^2},
\end{align*}
where the equality holds iff $a_2=1$. 
\item The case when $\delta_s^2-\delta_r\delta_{s+1}=0$ for all $s\in \{1,\ldots, n-1\}$ and $\delta_s\neq 0$ for some $s\in \{1,\ldots, n-1\}$. Let $s=s_0$ be the minimum number s.t. $\delta_{s_0}\neq 0$. Then, $\pm1= \delta_{s_0}=\delta_{s_0+1}=\delta_{s_0+2}=\cdots=\delta_{n}$. In particular, we have
\begin{align*}
\sum_{s=1}^n\delta_s N_s=\sum_{s=s_0}^n\delta_s N_s=\delta_{s_0}\sum_{s=s_0}^n (a_{n+1-s}-a_{n+2-s})=-\delta_{s_0}a_{n+2-s_0}.
\end{align*}
Therefore, \eqref{y2} implies
\begin{align*}
|y|^2=\frac{1}{R^2}+\frac{1}{R^2}a_{n+2-s_0}^2\geq \frac{2}{R^2},
\end{align*}
where the equality holds iff $a_{n+2-s_0}=1$. 
\end{itemize}

In particular, we see
\begin{align*}
d^2={\rm min}\{|y|^2; y\in \Gamma^*_{\bf a}\setminus\{0\}\}\geq \frac{2}{R^2},
\end{align*}
and the first eigenvalue of the Laplacian on $T^n_{\bf a}=T^{n+1}_{1/\sqrt{n+1}}/S^1$ w.r.t. the flat metric is estimated by
\begin{align*}
\lambda_1=4\pi^2d^2\geq 4\pi^2\frac{2}{R^2}=\frac{2}{r^2}.
\end{align*}
By the above argument, we see that the equality holds if and only if at least two components of $(1, a_2,\ldots, a_n)$ are equal. This completes the proof of Theorem \ref{clifford}.

\begin{remark}
By Proposition \ref{monotone}, $T^n_{\bf a}$ is a (unique) monotone Lagrangian torus orbit in $(\mathbb{C}P^n_{\bf a}, \omega_c)$. It was proved in \cite{AM}, $T^n_{\bf a}$ is Hamiltonian non-displaceable in $\mathbb{C}P^n_{\bf a}$ as well as the Clifford torus in $\mathbb{C}P^n$. 
\end{remark}

%% file: S4_glmcf.tex
\section{Generalized Lagrangian mean curvature flow}
 In the following, we always assume $(M, \omega, J)$ is a  compact K\"ahler manifold satisfying $\rho=C\omega+ndd^c f$ for some $f\in C^{\infty}(M)$. In this section, we introduce the generalized Lagrangian mean curvature flow in $M$ and show some fundamental properties of the flow. 

\subsection{Generalized Lagrangian mean curvature flow}
Let $\phi:L^n\rightarrow M^{2n}$ be a Lagrangian immersion of a compact manifold $L$ and consider a smooth deformation $F_t:L\times [0, T)\rightarrow M$ of the immersion $\phi$. In \cite{Beh}, Behrndt introduced the following initial value problem: 
\begin{align}
\begin{cases}\label{glmcf}
\dfrac{d}{dt}F_t=K=H-n(\Nab f)^\perp, \ \ t\in [0, T), \\
F_0=\phi. 
\end{cases}
\end{align}
Since this flow differs from usual mean curvature flow only by the lower oder term $-n(\Nab f)^\perp$, short time existence and uniqueness of the solution are valid if $L$ is compact. Moreover, Behrndt showed that if the initial immersion $F_0=\phi$ is Lagrangian, then $F_t$ is also Lagrangian for each $t\in [0, T)$ as long as the solution exists. On the other hand, the vector $K$ coincides with the mean curvature vector $H$ when $M$ is a K\"ahler--Einstein manifold. Thus, \eqref{glmcf} is a generalization of the Lagrangian mean curvature flow in K\"ahler--Einstein manifolds introduced by Smoczyk in \cite{Smo1} and \cite{Smo3}. From these facts, we call the flow \eqref{glmcf} \textit{Generalized Lagrangian mean curvature flow} (GLMCF for short).  By definition, stationary points of GLMCF are  $f$-minimal Lagrangians immersions. 

\subsection{Exact Lagrangian immersion}
In Subsection 2.4, we show the generalized mean curvature form $\alpha_K$ of a Lagrangian immersion is always a closed form. Therefore, there exists a function $\theta$ at least locally so that
 \begin{align*}
\alpha_K=d\theta. 
\end{align*}
If $\alpha_K$ is an exact form, then we take $\theta$ as a global smooth function on $L$, and we call $\phi:L\rightarrow M$ {\it exact} Lagrangian immersion. Moreover, the function $\theta$ will be called {\it Lagrangian angle} of $\phi$.

\begin{remark}
When $M$ is an almost Calabi-Yau manifold, the Lagrangian angle $\theta: L\rightarrow S^1$ for a Lagrangian immersion $\phi: L\rightarrow M$ is defined by
\begin{align*}
\phi^*\Omega=e^{-\sqrt{-1}\theta+nf}d\mu=e^{-\sqrt{-1}\theta}d\mu_f,
\end{align*}
where $\Omega$ is the 1-form defined by \eqref{holcf} (Note that our definition differs from \cite{Beh} up to sign).  This definition recovers our definition of Lagrangian angle (See Proposition 5.4 in \cite{Beh}), and the condition of exactness of $\alpha_K$ is called {\it zero-Maslov}. 
\end{remark}

By Proposition \ref{monotone}, the exactness of the mean curvature form is preserved under a Hamiltonian deformation $\phi_s$ of $\phi=\phi_0$.  Namely, we find a family of smooth functions $\{\theta_s\}_s$ on $L$ such that $\alpha_{K_s}=d\theta_s$ along a Hamiltonian deformation. 

 \begin{lemma} \label{lagvar}
Suppose $L_s=\phi_s(L)$ is a Hamiltonian deformation with $\alpha_{V_s}=du_s$, and $\alpha_{K_0}$ is exact. Then the Lagrangian angle $\theta_s$ satisfies 
\begin{align}
\frac{d}{ds}\theta_s=\Delta_fu_s+C u_s. \label{lagvar1}
\end{align}
\end{lemma}
\begin{proof}
Since $\alpha_K=d\theta$ is written by $K^i=\nab_i\theta$ in the normal coordinate frame. Similarly, $\alpha_{V}=du$ can be written as $V^i=\nab_iu$. Combining these with the evolution equation \eqref{dsaKF}, we have 
\begin{align*}
\nab_i\Big(\frac{d}{ds}\theta\Big)&=\frac{d}{ds}K^i\\
&=\nab_i(\nab_jV^j+nV^j\nab_jf)+C V^i\\
&=\nab_i(\Delta u+n\nab_j u\nab_jf)+C\nab_iu\\
&=\nab_i(\Delta_f u+C u). 
\end{align*}
This proves the formula.
\end{proof}

Let us consider GLMCF.  By using Lemma \ref{Ldsa}, we compute the following evolution equation for the generalized mean curvature form $\alpha_K$ along GLMCF: 
\begin{align}
\frac{d}{dt}\alpha_K=-d\delta_f\alpha_K+C\alpha_K, \ \ t\in[0, T). \label{dtaK}
\end{align}
Moreover, setting $\beta_t:=e^{-C t}\alpha_{K_t}$,  we see $\beta_t$ is also a closed form, and the evolution equation \eqref{dtaK} becomes 
\begin{align}
\frac{d}{dt}\beta_t=-d\delta_f \beta_t, \ \ t\in[0, T). \label{dtbK}
\end{align}
By a similar argument of the proof of Proposition \ref{monotone}, we obtain  the following as an extension of the result in \cite{Smo2}: 
\begin{proposition}\label{exact}
The cohomology class $[\beta_t]$ does not change under the GLMCF. In particular, the exactness of $\alpha_{K_0}$ is preserved under GLMCF, that is, if $\alpha_{K_0}$ is exact, then $\alpha_{K_t}$ is also exact for each $t\in [0, T)$.
\end{proposition}

We call a solution preserving the exactness of $F_t$ an \textit{exact solution} of GLMCF. Note that an exact solution to GLMCF generates a Hamiltonian deformation, although it is not true in general. Therefore, it is natural to ask whether the GLMCF of exact Lagrangian immersion converges to a $f$-minimal and Hamiltonian $f$-stable Lagrangian immersion. In the following sections, we shall consider this problem. 

At the end of this section, we give an evolutional equation of the Lagrangian angle for an exact solution: 

\begin{corollary}
For an exact solution to GLMCF, there exists a Lagrangian angle $\theta_t\in C^\infty(L_t)$ and it satisfies 
\begin{align}
\frac{d\theta_t}{dt}=\Delta_f\theta_t+C\theta_t, \ \ t\in[0, T).  \label{angle}
\end{align}
\end{corollary}
\begin{proof}
This is a direct consequence of Lemma \ref{lagvar} and Lemma \ref{exact}.  
\end{proof}

%% file: S5_estimate.tex
\section{Estimates under GLMCF}
This section is a preliminary to show the long-time existence and convergence results. We compute some evolution equations for fundamental quantities along the flow. 
\subsection{Weighted $L^2$ estimate}
First, we derive the evolution equation for $\int_{L_t}|K|^2\dmu_f$ under GLMCF. This is actually the (negative) second variation formula of $\Vol_f(F_t)$ with the variation direction of $V=K$ since, by the first variation formula, we know 
\begin{align*}
\frac{d}{dt}\Vol_{f}(F_t)=-\int_{L_t}|K|^2\dmu_f.  
\end{align*}

\begin{lemma} \label{dL2}
If $L_t$ is an exact solution to GLMCF with $\alpha_{K_t}=d\theta_t$ for $\theta_t\in C^\infty(L_t)$,  then we have
\begin{align}
\frac{d}{dt}\int_{L_t}|K|^2\dmu_f&=\int_{L_t}\big\{-2|\Delta_f\theta|^2+2C|K|^2+2g(B(JK, JK), K)-|K|^4\big\}\dmu_f\nonumber \\ 
&\leq 2\Big(C+\max_{L_t}|B||K|-\lambda_1(t)\Big)\int_{L_t}|K|^2\dmu_f. \label{L2e1}
\end{align}
In particular, it follows 
\begin{align}
\int_{L_t}|K|^2\dmu_f\leq e^{\int_0^t{2(C+\max_{L_{s}}(|K||B|)-\lambda_{1}(s))}ds} \int_{L_0}|K|^2\dmu_f.   \label{L2e3}
\end{align}
\end{lemma}
\begin{proof}
The first equality in \eqref{L2e1} follows from \eqref{var1}, \eqref{var2}, \eqref{first} and \eqref{second} using $K$ as the variation vector field. 

As for the inequality in \eqref{L2e1}, we use the relation 
\begin{align*}
|\Delta_f\theta|^2\geq \lambda_1 |\nab \theta|^2=\lambda_1|K|^2, 
\end{align*}
since exactness is equivalent to $\nab \theta=JK$. 
\end{proof}

\subsection{First eigenvalue estimates}
From the estimate \eqref{L2e3} for the weighted $L^2$ norm $\int_{L_t}|K|^2\dmu_f$, we expect that the flow converges exponentially fast to an $f$-minimal if we have a good control of $|B|, |K|$ and $\lambda_1(t)$ to keep 
\begin{align*}
C+\max_{L_{t}}(|K||B|)-\lambda_{1}(t)
\end{align*}
negative along the flow. Thus, we derive the first eigenvalue estimate here. Let $\varphi(x, t)\in C^\infty(L_t)$ be an eigenfunction which satisfies 
\begin{align}
\Delta_f \varphi+\lambda_{1}(t)\varphi=0, \ \ \ \int_{L_t}\varphi^2\dmu_f=1. \label{ef1}
\end{align} 

\begin{lemma}\label{ef}
Along GLMCF, the first eigenvalue  $\lambda_{1}(t)$ of the weighted Laplacian $\Delta_f$ satisfies 
\begin{align}
\frac{d}{dt}\lambda_{1}(t)&=\int_{L_t}\Big(2g(B(\nab \varphi, \nab \varphi), K)-|K|^2\big(|\nab \varphi|^2+\varphi(\Delta_f \varphi)\big)\Big)\dmu_f \nonumber \\
&\geq -c(n)\max_{L_t}(|B||K|+|K|^2)\lambda_{1}(t), \label{ef2}
\end{align}
for some dimensional constant $c(n)>0$. In particular, we have 
\begin{align}
\lambda_{1}(t)\geq e^{-E(t)}\lambda_{1}(0),   \label{eff}
\end{align}
where $E(t):=c(n)\int_0^t \max_{L_{s}}(|B||K|+|K|^2)ds>0$. Note that $E(t)$ may vary up to the dimensional constant $c(n)$. 
\end{lemma}
In the following, we prove Lemma \ref{ef}. First, taking time derivative of \eqref{ef1}, we have 
\begin{align}
\int_{L_t}\Big(2\varphi\frac{d\varphi}{dt}-\varphi^2|K|^2\Big)\dmu_f=0. \label{ef3}
\end{align}
By \eqref{ef1}, it is clear that the first eigenvalue satisfies 
\begin{align}
\lambda_{1}(t)=\int_{L_t}|\nab \varphi|^2\dmu_f. \label{ef4}
\end{align}
In order to compute the time derivative of the eigenvalue, we need the following commutation formula for the weighed Laplacian: 
\begin{claim}
Along GLMCF, we have
\begin{align}
\frac{d}{dt}\Big(\Delta_f \varphi\Big)-\Delta_f\Big(\frac{d\varphi}{dt}\Big)=-\delta_fX-\nab_i|K|^2\nab_i\varphi, \label{commf} 
\end{align}
where $X_i=2K^mh^m_{ij}\nab_j\varphi$ is a $1$-form on $L$. 
\end{claim}
\begin{proof}
We compute in the normal coordinate frame. For a time dependent function $\varphi(t)\in C^\infty(L_t)$, we know the following commutation formula under the generalized Lagrangian mean curvature flow: 
\begin{align}
\frac{d}{dt}\Big(\Delta \varphi\Big)-\Delta\Big(\frac{d}{dt}\varphi\Big)&=\Big(\frac{d}{dt}g^{ij}\Big)\nab_i\nab_j\varphi-g^{ij}\Big(\frac{d}{dt}\Gamma^k_{ij}\Big)\nab_k\varphi\label{comml}\\
&=2K^m h^m_{ij}\nab_i\nab_j \varphi+2\nab_i(K^m h^m_{ik})\nab_k\varphi-\nab_k(K^m H^m)\nab_k\varphi. \nonumber
\end{align}
Recall the definition of the weighted Laplacian $\Delta_f\varphi=\Delta \varphi+ng(\nab f, \nab \varphi)$ and take the time derivative of it. Then, by the commutation formula \eqref{comml} for usual Laplacian $\Delta$,  we have 
\begin{align}
\frac{d}{dt}\big(\Delta_f \varphi\big)=&\ \frac{d}{dt}\big(\Delta \varphi+ng(\nab f, \nab \varphi)\big)\nonumber\\ 
=&\ \Delta\Big(\frac{d\varphi}{dt}\Big)+2K^m h^m_{ij}\nab_i\nab_j \varphi+2\nab_i(K^m h^m_{ik})\nab_k\varphi-\nab_i(K^m H^m)\nab_i\varphi\label{com1}\\
&+2nK^mh^m_{ij}\nab_if\nab_j\varphi+n\nab_i\Big(\frac{df}{dt}\Big)\nab_i\varphi
+n\nab_if\nab_i\Big(\frac{d\varphi}{dt}\Big)\nonumber
\end{align}
Observe that $\Delta_f(d\varphi/dt)$ comes from the first and the last term in \eqref{com1}. Combining the second, third and fifth terms in \eqref{com1}, we obtain $-\delta_fX_i$. Note that the time derivative of $f$ becomes 
\begin{align*}
n\frac{df}{dt}=ndf(K)=ng(\Nab f, K)=g((\Nab f)^\perp, K). 
\end{align*}
Then the fourth and sixth terms become $-\nab_i|K|^2\nab_i \varphi$. This proves \eqref{commf}. 
\end{proof}
Now, we are ready to compute the time derivative of the first eigenvalue. Using the commutation formula \eqref{commf} and  \eqref{stokes}, we see 
\begin{align*}
\frac{d}{dt}\lambda_{1}(t)&=-\frac{d}{dt}\int_{L_t}\varphi(\Delta_f\varphi)\dmu_f\\ 
&=\int_{L_t}\Big\{-\frac{d\varphi}{dt}(\Delta_f\varphi)-\varphi\Delta_f\Big(\frac{d\varphi}{dt}\Big)+\varphi(\delta_fX)+\nab_i|K|^2\varphi\nab_i\varphi+\varphi(\Delta_f\varphi)|K|^2\Big\}\dmu_f\\
&=\int_{L_t}\Big\{-2\frac{d\varphi}{dt}(\Delta_f\varphi)+\varphi(\delta_fX)+\frac{1}{2}\nab_i|K|^2\nab_i\varphi^2+\varphi(\Delta_f\varphi)|K|^2\Big\}\dmu_f\\
&=\int_{L_t}\Big\{\lambda_{1}\Big(2\varphi\frac{d\varphi}{dt}-\varphi^2|K|^2\Big)+\varphi(\delta_fX)+\frac{1}{2}\nab_i|K|^2\nab_i\varphi^2\Big\}\dmu_f\\
&=\int_{L_t}\Big\{\varphi(\delta_fX)-\frac{1}{2}|K|^2(\Delta_f\varphi^2)\Big\}\dmu_f. 
\end{align*}
We have used the relation \eqref{ef3} for the last equality. Note that for the weighted Laplacian we see 
\begin{align*}
\Delta_f(\varphi^2)=2\varphi(\Delta_f \varphi)+2|\nab \varphi|^2. 
\end{align*}s
On the other hand, for the weighted co-differential $\delta_f$, we have 
\begin{align*}
\delta_f(\varphi X)=\varphi \delta_f X-g(d\varphi, X)=\varphi (\delta_f X)-2K^mh^m_{ij}\nab_i\varphi \nab_j\varphi. 
\end{align*}
Therefore, again by Lemma \ref{stokes}, we obtain the equality in \eqref{ef2}. The inequality part is easily follows from the relations \eqref{ef1} and \eqref{ef4}, then we complete the proof of Lemma \ref{ef}. 

\subsection{Zero and Higher order estimates}\label{smest}
In the following, we need pointwise estimates to show the convergence of GLMCF. Generalizing Li's computation in \cite{Li}, we list some estimates. Results in Section 3.3 and 3.4 of \cite{Li} for LMCF in K\"ahler-Einstein manifolds remain true in our case with slight modifications by a function $f\in C^\infty(M)$. 

First, we consider the higher order estimates for the second fundamental form $B$. In order to do so, we need the evolution equation of $|B|^2$. 
Using the $*$-notation, a similar computation to \cite{Wan} (see also \cite{SuY}) shows 
\begin{align}
\frac{d}{dt}B=&\ \Delta B+B*B*B+B*B*\Nab f \\
&+\oRm*B+\oRm*\Nab f+\Nab^3f+\Nab^2f*B+\Nab f*\nab B+\Nab \oRm.   \nonumber
\end{align}
More generally, by induction, we have the following general result (see \cite{HaS}, \cite{Smo4} and \cite{SuY} for the proof): 
\begin{lemma}\label{hderiv}
Let $L_t=F_t(L)$ be a compact GLMCF. Then along the flow, we have 
\begin{align*}
\frac{d}{dt}(\nab^mB)=&\ \Delta(\nab^m B)+\sum_{i+j+k=m}\nab^iB*\nab^jB*\nab^kB\\
&+\sum_{j=0}^m\sum_{i_1+\cdots+i_r+k=m+1-j}\Nab^j\oRm*\nab^{i_1-1}B*\cdots*\nab^{i_r-1}B*\Nab^kf\\
&+\sum_{k=0}^{m+2}\sum_{i_1+\cdots+i_r=m+2-k}\nab^{i_1-1}B*\cdots*\nab^{i_r-1}B*\Nab^{k+1}f\\  
&+\Nab^{m+1}\oRm. 
\end{align*}
Moreover, 
\begin{align}
\frac{d}{dt}|\nab^m B|^2\leq&\  \Delta|\nab^m B|^2-|\nab^{m+1}B|^2+c\sum_{i+j+k=m}|\nab^iB||\nab^jB||\nab^kB||\nab^m B|\label{horder1}\\
&+c_{m+1, m+3}\sum_{j=0}^{m+2}\sum_{\substack{i_1+\cdots+i_r=m+2-j\\i_1, \cdots, i_r<m+2}}|\nab^{i_1-1}B|\cdots |\nab^{i_r-1}B||\nab^m B|,  \nonumber
\end{align}
where $c$ is a constant, and $c_{k, r}$ is a constant depending only on 
\begin{align*}
\oR_k:=\sum_{i=0}^{k}\sup_{M}|\Nab^i\oRm|, \ \ f_r:=||f||_{C^7}(M)=\sum_{i=0}^{r} \sup_{M}|\Nab^if|. 
\end{align*}
\end{lemma}

Now, we are ready to show the higher order estimate for the second fundamental form $B$. \begin{lemma}\label{horder}
Assume $L_t=F_t(L)$ is compact and evolves under GLMCF. If 
\begin{align*}
\max_{L_t}|B|^2\leq \Lambda<\infty, \ \ t\in[0, T], \ \ T>0, 
\end{align*}
then for each $m\geq 1$, there exist constants $c_m=c_m(n, \Lambda, \oR_{m+1}, f_{m+3}, T)$ such that 
\begin{align*}
\max_{L_t}|\nab^m B|^2\leq \frac{c_m}{t^m}, \ \ t\in (0, T]. 
\end{align*}
As a direct consequence, we also have 
\begin{align*}
\max_{L_t}|\nab^m H|^2\leq \frac{c_m}{t^m}, \ \ \max_{L_t}|\nab^m K|^2\leq \frac{c_m}{t^m}, \ \ t\in (0, T]. 
\end{align*}
\end{lemma}
For the proof, see \cite{HaS} in the case of MCF. Although GLMCF case has to deal with the weight function $f$, a similar technique to \cite{HaS} also work well. Thus, we omit the proof. 

In \cite{LiY} and \cite{SuY}, by using the higher order estimate of $B$ and standard application of the maximum principle, they showed the extension result for GLMCF. 
\begin{proposition}\label{long}
If the second fundamental form of $L_t=F_t(L)$ is uniformly bounded under GLMCF for $t\in [0,T)$, then the solution can be extended beyond $T$. 
\end{proposition}

Next, we show that initial data do not change so much at least for a short time under GLMCF.
\begin{lemma}\label{unch}
If $L_0$ satisfies 
\begin{align*}
|B|(0)\leq \Lambda, \ \ |K|(0)\leq \epsilon, \ \ \lambda_1(0)\geq C+\delta, 
\end{align*}
then there exists $T=T(n, \Lambda, \oR_1, f_3)$ so that GLMCF satisfies 
\begin{align*}
|B|(t)\leq 2\Lambda, \ \ |K|(t)\leq 2\epsilon, \ \ \lambda_1(t)\geq C+\frac{2}{3}\delta, \ \  t\in[0, T]. 
\end{align*}
\end{lemma}
\begin{proof}
First, we show the Lemma for $|B|^2$. From \eqref{horder1} in Lemma \ref{hderiv}, we have 
\begin{align*}
\frac{d}{dt}|B|^2&\leq \Delta|B|^2+c_1|B|^4+c_2|B|^3+c_3|B|^2+c_4|B|\\
&\leq \Delta|B|^2+c_5|B|^4+c_6. 
\end{align*}
Note that we have used Young's inequality in the last line. Set 
\begin{align*}
t_0:=\sup\{s>0\ |\ |B|^2(t)\leq 4\Lambda, \ t\in[0, s)\}. 
\end{align*}
Then, for $t\in[0, t_0)$, we have 
\begin{align*}
\frac{d}{dt}|B|^2\leq \Delta|B|^2+(16\Lambda^4+1)c  
\end{align*}
for some constant $c>0$. 
Applying the maximum principle, we obtain 
\begin{align*} 
|B|^2(t)\leq \max_{L_0}|B|^2(0)+(16\Lambda^4+1)ct\leq 2\Lambda^2, \ \ \ 0\leq t\leq \frac{\Lambda^2}{(16\Lambda^4+1)c}. 
\end{align*}
Therefore $t_0$ satisfies 
\begin{align*}
t_0\geq \frac{\Lambda^2}{(16\Lambda^4+1)c}. 
\end{align*} 
Next, we prove the Lemma for $|K|$. A similar computation to \eqref{dsaKF} shows 
\begin{align*}
\frac{d}{dt}K^i=&\ \frac{d}{dt}(H^i-n(\Nab_if)^\perp)\nonumber\\
=&\ \Delta K^i+K^jh^{j}_{kl}h^i_{kl}-H^jK^kh^i_{jk}+K^j\oR_{\ui l\uj l} \\
&-n(\Nab_{\ui}\Nab_{\uj} f)K^j+n\nab_iK^j\nab_j f-nh^k_{ij}K^j\Nab_{\uk}f,  \nonumber 
\end{align*}
and then it follows  
\begin{align*}
\frac{d}{dt}|K|^2\leq&\ \Delta |K|^2+c|B||K|^3+c|H||B||K|^2+c(\oR_0)|K|^2+c|B|^2|K|^2\\
&+2ng(\nab f, \nab_{JK}(JK))+c(f_1)|B||K|^2. 
\end{align*}
Note that 
\begin{align*}
2ng(\nab f, \nab_{JK}(JK))&=ng(\nab f, \nab|K|^2), \\
|K|\leq |H|+n|\Nab f|&\leq\sqrt{n}|B|+c(f_1). 
\end{align*}
Therefore, for $t\leq t_0$, we have  
\begin{align*}
\frac{d}{dt}|K|^2\leq&\ \Delta |K|^2+ng(\nab f, \nab|K|^2)+(16\Lambda^2+4\Lambda+1)c|K|^2. 
\end{align*}
Thus, we can apply the maximum principle to obtain 
\begin{align*}
|K|^2(t)\leq |K|^2(0)e^{(16\Lambda^2+4\Lambda+1)ct}\leq 4\epsilon^2, \ \ t\in\bigg[0, \min\Big\{t_0, \frac{\log 4}{(16\Lambda^2+4\Lambda+1)c}\Big\}\bigg]. 
\end{align*}
Then, the Lemma for $|B|$ and $|K|$ follows if we choose $t_1$ as 
\begin{align*}
t_1=\min\bigg\{\frac{\Lambda^2}{(16\Lambda^4+1)c},\  \frac{\log 4}{(16\Lambda^2+4\Lambda+1)c} \bigg\}.  
\end{align*}
Finally, we show the lemma for $\lambda_1$. Recall the estimate \eqref{eff} for $\lambda_1$ in $[0, t_1]$, then we have 
\begin{align*}
\lambda_1(t)& \geq e^{-\int_0^{t_1}2\max(|B||K|+|K|^2)dt}\lambda_1(0)\\
&\geq e^{-(8\Lambda \epsilon+8\epsilon^2) t_1}. 
\end{align*}  
Hence we obtain 
\begin{align*}
\lambda_1(t)\geq C+\frac{2}{3}\delta, \ \ t\in[0, T], 
\end{align*}
for sufficiently small $T\leq t_1$. 
\end{proof}

Here, we consider the volume ratio on $L^n$. Let $B(x, s)\subset L^n$ be a geodesic ball centered at $x\in L$ with radius $s>0$. Since $L$ is compact, its injectivity radius $\inj(L)$ is bounded from below (See Proposition 14 in \cite{Cro}). Then there exist some positive constants $\kappa=\kappa(n, \inj(L))$ and $r=r(n, \inj(L))$ so that
\begin{align}\label{noncollapsing}
\frac{\Vol(B(x, s))}{s^n}\geq \kappa, \ \ \forall x\in L, \ \ 0<s\leq r.
\end{align}
This volume ratio condition is called \textit{$\kappa$-noncollapsed on the scale $r$}. 
In order to see the control of the volume ratio condition, we need to check the change of the injectivity radius of $L$. 
\begin{lemma}\label{inje}
Assume that GLMCF $L_t$ satisfies 
\begin{align*}
|B|(t)\leq \Lambda<\infty, \quad t\in [0, T], 
\end{align*}
then the injectivity radius of $L_t$ is uniformly bounded from below: 
\begin{align}
\inj(L_t)\geq \iota>0, \quad t\in (0, T]
\end{align}
by some positive constant $\iota=\iota(n, \Lambda, \oR_0, \inj(M))>0$. Moreover there exist time-independent constants $\kappa, r>0$ such that $L_t$ is $\kappa$-noncollapsed on the scale $r$ along GLMCF in $t\in (0, T]$. 
\end{lemma}
\begin{proof}
Combining Lemma \ref{horder} above and Proposition 2.2 in \cite{ChH}, we first obtain the uniform bound for $\inj(L_t)\geq \iota>0$ in $(0, T]$. If we choose $r>0$ small enough such that $0< r\leq \frac{1}{2}\iota$, then we can apply Proposition 14 in \cite{Cro} to conclude that $L_t$ is $\kappa$-noncollapsed on the scale $r$ for some constant $\kappa>0$ in $(0, T]$.  
\end{proof}

The following Lemma tells us the change of the volume ratio condition from the initial data. 
\begin{lemma} \label{kpnon}
Suppose that $L_0$ is $\kappa$-noncollapsed on the scale $r$. Then along GLMCF, it follows 
\begin{align*}
\frac{\Vol(B_t(x, s))}{s^n}\geq \kappa e^{-E(t)}, 
\end{align*}
where $E(t)>0$ is the function given in Lemma \ref{ef}, namely, 
\begin{align*}
E(t)=c(n)\int_0^t \max_{L_{s}}(|B||K|+|K|^2)ds>0. 
\end{align*}
\end{lemma}


Next result lead us to get a $C^0$ estimate from an $L^2$ estimate. 
\begin{lemma}\label{zero}
Suppose $L\subset M$ is compact and $\kappa$-noncollapsed on the scale $r$. For any tensor $S$ on $L$, if 
\begin{align*}
|\nab S|\leq \Lambda\  \textup{\it and}\ \int_L|S|^2\dmu_f\leq \epsilon\leq r^{n+2}, 
\end{align*}
then it follows 
\begin{align*}
\max_L|S|\leq \Big(\frac{1}{\sqrt{\kappa\cdot \min_M e^{nf}}}+\Lambda \Big)\epsilon^{\frac{1}{n+2}}=c\epsilon^{\frac{1}{n+2}},  
\end{align*} 
where $c=c(n, \Lambda, f_0, \inj(M))>0$ is a constant. 
\end{lemma}
The proofs for the above two Lemmas can be found in Lemma 3.4 and Lemma 3.5 in \cite{Li}.


%% file: S67_converge.tex
\section{Convergence result I} 
In this section, we prove the main results by generalizing Li's argument in \cite{Li} with the weight $f$. 
Our main result is the following: 
\begin{theorem}\label{cvthm}
Let $(M, \omega, J)$ be a compact K\"ahler manifold satisfying $\rho=C\omega+ndd^cf$. Suppose $L$ is a compact and exact Lagrangian submanifold which is smoothly immersed into $M$. For any $V_0, \Lambda_0$ and $\delta_0>0$, there exists $\epsilon_0=\epsilon_0(n, V_0, \Lambda_0, \delta_0,  C, \oR_5, f_7, \inj(M))$ such that if $L$ satisfies 
 \begin{align*}
\Vol_f(L)\leq V_0, \ \ |B|\leq \Lambda_0,  \ \ \lambda_{1}(\Delta_f)\geq C+\delta_0, \ \ \int_{L}|K|^2\dmu_f\leq\epsilon_0, 
\end{align*}
then the generalized mean curvature flow with initial data $L$ converge exponentially fast to an $f$-minimal Lagrangian in $M$. 
\end{theorem}

\begin{proof}
First, we note that $L$ is $\kappa_1$-noncollapsed on the scale $r_1$ for some $\kappa_1>0$ and $r_1>0$ since $L$ is compact. From Lemma \ref{unch}, there exists $T_0>0$ so that  
\begin{align}
|B|(t)\leq 2\Lambda_0, \ \ \lambda_1(t)\geq C+\frac{2}{3}\delta_0, \ \ t\in[0, T_0],  
\end{align}
and by Lemma \ref{inje}, there exist constants $\kappa_2, r_2>0$ such that $L_t$ is $\kappa_2$-noncollapsed on the scale $r_2$ in $(0, T_0]$ along GLMCF with initial data $L_0=L$. Set $\kappa_0=\min\{\kappa_1, \kappa_2\}>0$ and $r_0=\min\{r_1, r_2\}>0$. Then $L_t$ is $\kappa_0$-noncollapsed on the scale $r_0$ for $t\in [0, T_0]$.

By the estimate \eqref{L2e3}, there exists $c=c(n, \Lambda_0, \delta_0, C)>0$ so that 
\begin{align} 
\int_{L_t}|K|^2\dmu_f\leq e^{ct}\int_{L_0}|K|^2\dmu_f\leq e^{ct}\epsilon_0, \ \ t\in [0, T_0]. 
\end{align}
Therefore if we choose $t_0<T_0$ small enough, it follows 
\begin{align}
\int_{L_t}|K|^2\dmu_f\leq 2\epsilon_0, \ \ t\in [0, t_0]. 
\end{align}
Using Lemma \ref{horder} to obtain 
\begin{align}
|\nab B|(t)\leq c, \ \ t\in \bigg[\frac{1}{2}t_0, t_0\bigg], 
\end{align}
where $c=c(n, \Lambda_0, \oR_2, f_4)>0$ is a uniform constant. By Lemma \ref{zero}, there exists $c=c(n, \Lambda_0, \oR_2, f_4, \kappa_0, r_0)>0$ so that 
\begin{align}
|K|(t)\leq c\epsilon_0^{\frac{1}{n+2}}, \ \ t\in\bigg[\frac{1}{2}t_0, t_0\bigg], 
\end{align}
for $2\epsilon_0\leq r_0^{n+2}$. Now, regard $t=\frac{1}{2}t_0$ as the initial time of GLMCF. Then, without loss of generality, we may assume $L_0$ is $\kappa_0$-noncollapsed on the scale $r_0$ and 
\begin{align}
\Vol_f(L_0)\leq V_0, \ \ |B|(0)\leq \Lambda, \ \ \lambda_1(0)\geq C+\delta, \ \ |K|(0)\leq \epsilon, \label{initialdata}
\end{align}
where $\Lambda:=2\Lambda_0, \delta:=\frac{2}{3}\delta_0$ and $\epsilon:=c\epsilon_0^{\frac{1}{n+2}}$. 

In the following, we consider GLMCF with initial data \eqref{initialdata}. Again by Lemma \ref{unch}, Lemma \ref{inje} and Lemma \ref{kpnon}, if we choose $T>0$ and $\epsilon<1$ small enough, then  
$L_t$ is $\frac{1}{3}\kappa_0$-noncollapsed on the scale $r_0$ and 
 satisfies 
\begin{align*}
(*)\ \ \ \ |B|(t)\leq 6\Lambda, \ \ \lambda_1(t)\geq C+\frac{1}{3}\delta, \ \ |K|(t)\leq  2\epsilon^{\frac{1}{n+2}}, 
\end{align*} 
for $t\in [0, T]$. To show the long-time existence of GLMCF by contradiction, assume that  $[0, T_*), T_*<\infty$ be the maximal time interval in which $L_t$ is  $\frac{1}{3}\kappa_0$-noncollapsed on the scale $r_0$ and satisfies $(*)$. 

\begin{claim} \label{keyclaim}
There exists $\epsilon=\epsilon(n, \kappa, r_0, \delta, \Lambda, V_0, C, \oR_5, f_7, V_0) >0$ so that if $L_t$ is $\frac{1}{3}\kappa_0$-noncollapsed on the scale $r_0$ and satisfies $(*)$, then it follows that $L_t$ is $\frac{2}{3}\kappa_0$-noncollapsed on the scale $r_0$ and satisfies 
\begin{align*}
(**)\ \ \ \ |B|(t)\leq 3\Lambda, \ \ \lambda_1(t)\geq C+\frac{2}{3}\delta, \ \ |K|(t)\leq \epsilon^{\frac{1}{n+2}},  
\end{align*}
for $t\in[0, T_*)$. 
\end{claim} 
Then, however, by Lemma \ref{unch}, Lemma \ref{inje} and Lemma \ref{kpnon} there exists $T_+>0$ such that $L_t$ is again $\frac{1}{3}\kappa_0$-noncollapsed on the scale $r_0$ and satisfies $(*)$ for $t\in [0, T_*+T_+]$. This contradicts the maximality of $T_*<\infty$, so that $T_*=\infty$, and the long-time existence of GLMCF follows immediately.  

Now, we give a proof of Claim \ref{keyclaim}. First, we show the claim about $|K|(t)$. By the assumption $(*)$, the first eigenvalue  satisfies $\lambda_{1}(t)\geq C+\delta/3$. Thus, if we choose $\epsilon>0$ small enough, it follows 
\begin{align*}
\lambda_{1}(t)\geq C+\frac{\delta}{4}+6\Lambda\cdot 2\epsilon^{\frac{1}{n+2}}, \ \ t\in [0, T_*). 
\end{align*} 
Therefore, by the $L^2$ estimate \eqref{L2e3}, the initial condition \eqref{initialdata} and $(*)$, we have 
\begin{align}
\int_{L_t} |K|^2\dmu_f\leq e^{-\frac{\delta_0}{2}t}\int_{L_0}|K|^2\dmu_f\leq c(n, f_0) V_0\epsilon^2e^{-\frac{\delta_0}{2}t}, \ \ t\in[0, T_*). \label{0T}
\end{align}
Now, using Lemma \ref{unch}, there exists small $0<\tau<T_*$ such that 
\begin{align}
|K|(t)\leq 2\epsilon\leq \epsilon^{\frac{1}{n+2}}, \ \ t\in [0, \tau],  \label{a}
\end{align}
for $\epsilon\leq \frac{1}{2}$. 
On the other hand, by the assumption $(*)$, we have the bound $|B|(t)\leq 6\Lambda$ in $[0, T_*)$. Thus, Lemma \ref{horder} with $\tau>0$ implies the uniform bound  
\begin{align*}
|\nab B|(t)\leq c(n, \Lambda, \oR_2, f_4), \ \ t\in [\tau, T_*). 
\end{align*} 
It then follows $|\nab K|(t)\leq c$ for $t\in [\tau, T_*)$. Combining this with \eqref{0T}, by Lemma \ref{zero}, we obtain 
\begin{align*}
|K|(t)\leq c(n, \kappa, r_0, \Lambda, \oR_2, f_4, V_0)\epsilon^{\frac{2}{n+2}}e^{-\frac{\delta_0}{2(n+2)}t}, \ \ t\in[\tau, T_*).  
\end{align*}
Moreover, we can choose $\epsilon$ small such that 
\begin{align*}
c(n, \kappa, r_0, \Lambda, \oR_2, f_4, V_0)\epsilon^{\frac{1}{n+2}}\leq 1, 
\end{align*}
then we have 
\begin{align}
|K|(t)\leq \epsilon^{\frac{1}{n+2}}e^{-\frac{\delta_0}{2(n+2)}t}\leq \epsilon^{\frac{1}{n+2}}, \ \ t\in[\tau, T_*). \label{b}
\end{align}
Combining \eqref{a} and \eqref{b}, we obtain the estimate $|K|(t)\leq \epsilon^{\frac{1}{n+2}}$ in $[0, T_*)$. 

Next we prove the claim about $|B|(t)$. As in the proof for \eqref{b}, by Lemma \ref{horder}, there exists $c_m=c_m(n, \Lambda, \oR_{m+1}, f_{m+3})$ such that 
\begin{align*}
|\nab^m B|\leq c_m, \ \ t\in[\tau, T_*), 
\end{align*} 
thus, it also follows $|\nab^m K|\leq c_m$ for $t\in[\tau, T_*)$. Combining this with the (usual) divergence theorem and \eqref{b}, we have 
\begin{align*}
\int_{L_t}|\nab^2K|^2\dmu \leq\int_{L_t}|K||\nab^4 K|\dmu \leq c(n, \Lambda, \oR_{5}, f_{7}, V_0)\epsilon^{\frac{1}{n+2}}e^{-\frac{\delta_0}{2(n+2)}t}. 
\end{align*} 
for $t\in[\tau, T_*)$. Hence, we can apply Lemma \ref{zero} (without density) to $\nab^2 K$, then it follows 
\begin{align}
|\nab^2 K|\leq c(n, \kappa, r_0, \Lambda, \oR_5, f_7, V_0)\epsilon^{\frac{1}{(n+2)^2}}e^{-\frac{\delta_0}{2(n+2)^2}t}, \ \ t\in[\tau, T_*).  \label{nab2K}
\end{align}
By \eqref{dth}, we have 
\begin{align}
\frac{d}{dt}|B|\leq |\nab^2 K|+c|B|^2|K|+c(\oR_0)|K|. \label{dtBr}
\end{align}
Inserting \eqref{nab2K} into \eqref{dtBr} and integrating it w.r.t. $t$ from $\tau$ to $t$, by \eqref{b}, we have 
\begin{align*}
|B|(t)\leq&\  |B|(\tau)+\int_{\tau}^t|\nab^2 K|+(c+|B|^2)|K|ds\\ 
\leq&\  2\Lambda+c\epsilon^{\frac{1}{(n+2)^2}}\frac{2(n+2)^2}{\delta_0}+\epsilon^{\frac{1}{n+2}}\frac{2(n+2)}{\delta_0}. 
\end{align*}
Therefore, if we choose $\epsilon>0$ small enough, then we have 
\begin{align}
|B|(t)\leq 3\Lambda, \ \ t\in[0, T_*). \label{pc}
\end{align}

Then, we will show the estimate for $\lambda_1(t)$. Since $\lambda_1(0)\geq C+\delta$, using Lemma \ref{unch}, we have 
\begin{align}
\lambda_1(t)\geq C+\frac{2}{3}\delta, \ \ t\in[0, \tau] \label{pd}
\end{align}
for small $0<\tau<T_*$. On the other hand, the property \eqref{b}, \eqref{pc} and the estimate \eqref{eff} implies 
\begin{align}
\lambda_{1}(t)\geq e^{-c(n, \Lambda_0, \delta_0)\big(\epsilon \tau+\epsilon^2\tau+\epsilon^{\frac{1}{n+2}}+\epsilon^{\frac{2}{n+2}}\big)}\lambda_{1}(\tau), \ \ t\in[\tau, T_*). \label{pe}
\end{align}
Combining \eqref{pd} and \eqref{pe}, if $\epsilon>0$ is sufficiently small, we obtain 
\begin{align*}
\lambda_{1}(t)\geq C+\frac{\delta_0}{2}, \ \ t\in[0, T_*). 
\end{align*}

As for the estimate of $\kappa$-noncollapsing, we use the same method as the estimate for $\lambda_1(t)$. By Lemma \ref{inje} and Lemma \ref{kpnon} combining the properties \eqref{b} and \eqref{pc}, we obtain desired estimate if we choose $\epsilon>0$ small enough. This proves Claim \ref{keyclaim}.

Finally, we show the convergence of GLMCF to an $f$-minimal Lagrangian. But this is immediately follows from the property \eqref{b} in the proof of Claim \ref{keyclaim}: 
\begin{align*}
|K|(t)\leq \epsilon^{\frac{1}{n+1}}e^{-c(n, \delta)t}, \ \ t\in[\tau, \infty). 
\end{align*}
This completes the proof of  Theorem \ref{cvthm}.  
\end{proof}

\section{Convergence result II}
The assumption $\lambda_1\geq C+\delta_0$ for some $\delta_0>0$ in the previous Theorem is satisfied for an initial Lagrangian which is sufficiently close to a Hamiltonian $f$-stable Lagrangian with $\lambda_1>C$. For example, a small Hamiltonian deformation of some examples given in Theorem \ref{clifford} satisfies  assumptions of Theorem \ref{cvthm}. However, Theorem \ref{clifford} shows that there exist an example of $f$-minimal and Hamiltonian $f$-stable Lagrangian with $\lambda_1=C$, and we cannot apply Theorem \ref{cvthm} to an initial Lagrangian which is sufficiently close to such an example. In the last section, we consider this latter case following the technique by \cite{Li} again. 

Let $(M, \omega, J)$ be a K\"ahler manifold satisfying $\rho=C\omega+ndd^cf$ for some $f\in C^\infty(M)$. Suppose that $M$ is compact and $C> 0$ as before. 

In the following, we assume $\phi_0:L\rightarrow M$ is a compact $f$-minimal Lagrangian, and $X$ is a Hamiltonian variation vector field along $L_0=\phi_0(L)$, that is $JX=\nab u_0$. Let $\phi_s: L\rightarrow M$ be a Hamiltonian deformation of $L_0$  satisfying $(d/ds)|_{s=0}\phi_s=X$. We write $L_s=\phi_s(L_0)$. Since $\phi_0$ is $f$-minimal, $\phi_s$ is an exact Lagrangian immersion for any $s$ by Proposition \ref{monotone}. Therefore, it is natural to ask whether the GLMCF $F_t$ with initial data $F_0=\phi_s$ converges to a $f$-minimal and Hamiltonian $f$-stable Lagrangian.
 In order to describe a precise statement, we assume the Hamiltonian deformation $\phi_s$ satisfies
\begin{align}
\frac{d^2}{ds^2}\Big|_{s=0}\Vol_f(\phi_s)>0, \label{essHam}
\end{align}
otherwise, $\phi_s$ may be a volume preserving deformation of $\phi_0$ and it does not move under the flow. We call the deformation $\phi_s$ satisfying \eqref{essHam} an \textit{essential Hamiltonian deformation}. 
 By the same way as  Li \cite{Li}, the essential Hamiltonian deformation is characterized as follows:
\begin{lemma}\label{essential}
$\phi_s$ is an essential Hamiltonian deformation of $L_0$ if and only if the Hamiltonian function $u_0$ of the variation vector field $X$ along $L_0$ satisfies $u_0\not\in E_{\lambda_{1}}$, where $E_{\lambda_{1}}$ is the first eigenspace of $\Delta_f$ on $L_0$.  
\end{lemma}

In order to apply the technique by Li \cite{Li} to our convergence result, we need the following compactness result. The proposition is originally proved by Chen-He \cite{ChH} for MCF in the Euclidean space. Li \cite{Li} pointed out that the compactness result is also valid in general compact ambient manifolds since the manifold is isometrically embedded into some Euclidean space by the embedding theorem, and the corresponding second fundamental forms are still uniformly bounded. 
\begin{proposition}[cf. \cite{ChH}, \cite{Li}] \label{cptmcf}
Let $\phi_k(t):L\rightarrow M$ be a sequence of generalized Lagrangian mean curvature flow from a compact submanifold $L$ into a compact K\"ahler manifold satisfying $\rho=C\omega+ndd^cf$ with uniformly bounded second fundamental forms 
\begin{align*}
|B_k|(t)\leq C, \ \ \forall t\in [0, T]. 
\end{align*} 
Then there exists a subsequence of $\phi_k(t)$ which smoothly converges to a generalized Lagrangian mean curvature flow $\phi_\infty(t)$ for each $t\in(0, T)$ in the geometric sense, and $L_\infty=\phi_\infty(t)(L)$ is a smooth Riemannian manifold. 
\end{proposition} 
\begin{proof}
Similar to the usual MCF, we have uniform bounds of all the higher order derivatives of the second fundamental forms for GLMCF as in subsection \ref{smest} because $|B_k|(t)\leq C$. Likewise, since each $\phi_k(t)$ satisfies GLMCF, all the time derivatives are also uniformly bounded. Then the proposition follows from the same argument in \cite{ChH} and \cite{Li}. 
\end{proof}

Finally, we consider GLMCF starting from $\phi_s(L_0)$, where $L_0\subset M$ is an $f$-minimal Lagrangian with $\lambda_1=C$, and $\phi_s$ is a small Hamiltonian deformation of $L_0$.  We denote by $L_{s, t}=\phi_{s,t}(L_0), (t\in[0, T])$ the GLMCF with initial data $L_s=\phi_s(L_0)$. Since $L_0$ is $f$-minimal, $\alpha_{K_{s,0}}$ of $L_{s,0}$ is exact by Proposition \ref{monotone}.  Also by Lemma \ref{exact}, $\alpha_{K_{s,t}}$ on $L_{s, t}$ is exact, and we denote the Lagrangian angle by $\theta_{s,t}$. Suppose that the Hamiltonian deformation $\phi_s$ is sufficiently close to $\phi_0$ in the following sense: 
\begin{align}
||\phi_s-\phi_0||_{C^3}\leq \epsilon_0 \label{smdef}
\end{align}
for small $\epsilon_0>0$ which will be determined later. Then we have the following:  
\begin{lemma} \label{esse}
Let $L_0$ be an $f$-minimal Lagrangian with $\lambda_1=C$, and $X$ be an essential Hamiltonian variation along $L_0$ with $JX=\nab u_0$.  For any $\Lambda>0$, there exists $\epsilon_0=\epsilon_0(L_0, X, M)>0$ and $\delta_0>0$ such that if $L_{s,t}$ satisfies 
\begin{align}
|B_s|(t)\leq \Lambda, \ \ |K_s|(t)\leq \epsilon_0, \ \ \forall t\in [0, T], 
\end{align} 
then the Lagrangian angle $\theta_{s,t}$ of $L_{s,t}$ satisfies 
\begin{align}
\int_{L_{s,t}}|\Delta_f\theta_{s,t}|^2\dmu_f\geq (C+\delta_0)\int_{L_{s,t}}|\nab\theta_{s,t}|^2\dmu_f, \ \ \forall t\in[0, T]. 
\end{align}
Moreover, we have 
\begin{align}
\frac{d}{dt}\int_{L_{s,t}}|K_{s,t}|^2\dmu_f \leq -2(\delta_0-\Lambda\epsilon_0)\int_{L_{s,t}}|K_{s,t}|^2\dmu_f, \ \ \forall t\in [0, T]. 
\end{align}
\end{lemma}
\begin{proof}
We generalize the proof of Lemma 6.6 in \cite{Li}. The only thing we have to do is to translate the argument in \cite{Li} into our case by using Proposition \ref{2ndvf}, variation formulas \eqref{lagvar1}, \eqref{angle} for the Lagrangian angle, Lemma \ref{essential} and Proposition \ref{cptmcf}.  
\end{proof}

Finally, we obtain the following long-time existence and convergence result starting from the initial data sufficiently close to an $f$-stable Lagrangian. This extend the result by \cite{Li}. 
\begin{theorem} \label{main2}
Let $(M, \omega, J)$ be a compact K\"ahler manifold satisfying $\rho=C\omega+ndd^cf$ with $C>0$. 
Suppose that $\phi: L\rightarrow M$ is a compact $f$-minimal Lagrangian submanifold with $\lambda_1=C$, and $\phi_s$ is an essential Hamiltonian variation of $L_0=\phi(L)$ with $X$ as the variation vector field. Then there exists $\epsilon_0=\epsilon_0(X, L_0, M)>0$ such that if $L_s=\phi_s(L)\subset M$ satisfies  
\begin{align*}
||\phi_s-\phi_0||_{C^3}\leq \epsilon_0,  
\end{align*}
then GLMCF with the initial Lagrangian $L_s$ will converge exponentially fast to an $f$-minimal Lagrangian in $M$. 
\end{theorem}
\begin{proof}
Thanks to the Lemma \ref{esse}, the same argument as the proof of Theorem \ref{cvthm} is also valid for this case. See also the proof of Theorem 1.4 in \cite{Li} (Section 6). 
\end{proof}